\documentclass[oneside]{amsart}

\usepackage{amsthm}
\usepackage{amsmath}
\usepackage{amssymb}
\usepackage{enumerate}
\usepackage{graphicx}
\usepackage[hidelinks,pagebackref,pdftex]{hyperref}
\usepackage{booktabs}
\usepackage{color}
\usepackage[dvipsnames]{xcolor}
\usepackage{import}
\usepackage{tikz-cd}

\AtBeginDocument{%
   \def\MR#1{}
}

\usepackage{marginnote}
\long\def\@savemarbox#1#2{\global\setbox#1\vtop{\hsize\marginparwidth 
  \@parboxrestore\tiny\raggedright #2}}
\renewcommand*{\backref}[1]{}
\renewcommand*{\backrefalt}[4]{
  \ifcase #1
  [No citations.]
  \or [#2]
  \else [#2]
  \fi }

\numberwithin{equation}{section}
\theoremstyle{plain}
\newtheorem{theorem}[equation]{Theorem}
\newtheorem{corollary}[equation]{Corollary}
\newtheorem{lemma}[equation]{Lemma}
\newtheorem{conjecture}[equation]{Conjecture}

\newtheorem*{namedtheorem}{\theoremname}
\newcommand{\theoremname}{testing}

\theoremstyle{definition}
\newtheorem{definition}[equation]{Definition}

\newcommand{\cut}{\backslash\backslash}

\def\rdots{\rotatebox[origin=l]{29}{$\scriptscriptstyle\ldots\mathstrut$}}
\title{Algorithms in 3-manifold theory}
\author{Marc Lackenby}
\address{Mathematical Institute, University of Oxford, \newline Woodstock Road, Oxford OX2 6GG, United Kingdom}

\begin{document}


\maketitle
\tableofcontents

\section{Introduction}\label{Sec:Intro}

One of the revolutions in twentieth century mathematics was the discovery by Church \cite{Church} and Turing \cite{Turing} that there are fundamental limits to our understanding of various mathematical objects. An important example of this phenomenon was the theorem of Adyan \cite{Adyan} that one cannot decide whether a given finitely presented group is the trivial group. In some sense, this is a very negative result, because it suggests that there will never be a full and satisfactory theory of groups. The same is true of manifolds in dimensions 4 and above, by work of Markov \cite{Markov}. However, one of the main themes of low-dimensional topology is that compact 3-manifolds are tractable objects. In particular, they can be classified, since the homeomorphism problem for them is solvable. So the sort of wildness that one encounters in group theory and higher-dimensional manifold theory is not present in dimension 3. In fact, 3-manifolds are well-behaved objects, in the sense that many algorithmic questions about them are solvable and many optimistic conjectures about them have been shown to be true.

However, our understanding of 3-manifolds is far from complete. The status of the homeomorphism problem is a good example of this. Although we can reliably decide whether two compact 3-manifolds are homeomorphic \cite{KuperbergAlgorithmic, ScottShort}, the known algorithms for achieving this might take a ridiculously long time. If the 3-manifolds are presented to us by two triangulations, then the best known running time for deciding whether they are homeomorphic is, as a function of the number of tetrahedra in each triangulation, a tower of exponentials \cite{KuperbergAlgorithmic}. This highlights a general rule: although 3-manifolds are tractable objects, they are only just so. It is, in general, not at all straightforward to probe their properties and typically quite sophisticated tools are required. But there is an emerging set of techniques that do lead to more efficient algorithms. For example, the problem of recognising the unknot is now known \cite{HassLagariasPippenger, LackenbyEfficientCertification} to be in the complexity classes NP and co-NP. Thus, there are ways of certifying in polynomial time whether a knot diagram represents the unknot or whether it represents a non-trivial knot. But whether this problem is solvable using a deterministic polynomial-time algorithm remains unknown.

My goal in this article is to present some of the known algorithms in 3-manifold theory. I will highlight their apparent limitations, but I will also present some of the new techniques which lead to improvements in their efficiency. My focus is on the theoretical aspects of algorithms about 3-manifolds, rather than their practical implementation. However, it would be remiss of me not to mention here the various important programs in the field, including Snappea \cite{Weeks, CullerDunfield} and Regina \cite{BurtonRegina}.

Needless to say, this survey is far from complete. I apologise to any researchers whose work has been omitted. Inevitably, there is space only to give sketches of proofs, rather than complete arguments. The original papers are usually the best places to look up these details. However, Matveev's book \cite{Matveev} is also an excellent resource, particularly for the material on normal surfaces in Sections \ref{Sec:NormalSurfaces}, \ref{Sec:MatchingFundamental} and \ref{Sec:Hierarchies}. In addition, there are some other excellent surveys highlighting various aspects of algorithmic 3-manifold theory, by Hass \cite{HassSurvey}, Dynnikov \cite{DynnikovSurvey} and Burton \cite{BurtonSurvey}.

I would like to thank the referee and Mehdi Yazdi for their very careful reading of the paper and for their many helpful suggestions.

\section{Algorithms and complexity}

\subsection{Generalities about algorithms} 

This is not the place to give a detailed and rigorous introduction to the theory of algorithms. However, there are some aspects to the theory that are perhaps not so obvious to the uninitiated.

An algorithm is basically just a computer program. The computer takes, as	 its input, a finite string of letters, typically 0's and 1's. It then starts a deterministic process, which involves passing through various states and reading the string of letters that it is given. It is allowed to write to its own internal memory, which is of unbounded size. At some point, it may or may not reach a specified terminating state, when it declares an output. In many cases, this is just a single 0 or a single 1, which is to be interpreted as a `no' or `yes'.

Algorithms are supposed to solve practical problems called decision problems. These require a `yes' or `no' answer to a specific question, and the algorithm is said to solve the decision problem if it reliably halts with the correct answer.

One can make all this formal. For example, it is usual to define an algorithm using Turing machines. This is for two reasons. Firstly, it is intellectually rigorous to declare at the outset what form of computer one is using, in order that the notion of an algorithm is well-defined. Secondly, the simple nature of Turing machines makes it possible to prove various non-existence results for algorithms.
However, in the world of low-dimensional topology, these matters do not really concern us. Certainly, we will not describe any of our algorithms using Turing machines. But we hope that it will be evident that all the algorithms that we describe could be processed by a computer, given a sufficiently diligent and patient programmer.

However, this informal approach can hide some important points, including the following:

(1) Decision problems are just functions from a subset of the set of all finite strings of 0's and 1's to the set $\{ 0, 1 \}$. In other words, they provide a yes/no answer to certain inputs. So, a question such as `what is the genus of a given knot?' is not a decision problem. One can turn it into a decision problem by asking, for instance, `is the genus of a given knot equal to $g$?' but this changes things. In particular, a fast solution to the second problem does not automatically lead to a fast solution to the first problem, since one would need run through different values of $g$ until one had found the correct answer.

(2) Typically, we would like to provide inputs to our programs that are not strings of 0's and 1's. For example, some of our inputs may be positive integers, in which case it would be usual to enter them in binary form. However, in low-dimensional topology, the inputs are usually more complicated than that. For example, one might be given a simplicial complex with underlying space that is a compact 3-manifold. Alternatively, one might be given a knot diagram. Clearly, these could be turned into 0's and 1's in some specific way, so that they can be fed into our hypothesised computer. 

(3) However, when discussing algorithms, it is very important to specify what form the input data takes. Although it is easy to encode knot diagrams as 0's and 1's, and it is easy to encode triangulations using 0's and 1's, it is not at all obvious that one can easily convert between these two forms of data. For example, suppose that you are given a triangulation of a 3-manifold and that you are told that it is the exterior of some knot. How would you go about drawing a diagram of this knot? It turns out that it is possible, but it is still unknown how complex this problem is.

(4) Later we will be discussing the `complexity' of algorithms, which is typically defined to be their longest possible running time as a function of the length of their input. Again, the encoding of our input data is important here. For example, what is the `size' of a natural number $n$? In fact, it is usual to say that the `size' of a positive integer $n$ is its number of digits in binary. We will always follow this convention in this article. On the other hand, the `size' of a triangulation $\mathcal{T}$ of a 3-manifold is normally its number of tetrahedra,
which we will denote by $|\mathcal{T}|$.

(5) Algorithms and decision problems are only interesting when the set of possible inputs is infinite. If our decision problem has only a finite number of inputs, then it is always soluble. For instance, suppose that our problem is `does this specific knot diagram represent the unknot?'. Then there is a very short algorithm that gives the right answer. It is either the algorithm that gives the output `yes' or the algorithm that gives the output `no'. This highlights the rather banal point that, in some sense, we do not care how the computer works, as long as it gives the right answer. But of course, we do care in practice, because the only way to be sure that it is giving the right answer is by checking how it works.

\subsection{Complexity classes}
\label{Sec:ComplexityClasses}

\begin{definition}
\begin{enumerate}
\item A decision problem lies in $\mathrm{P}$, or runs in \emph{polynomial time}, if there is an algorithm to solve it with running time that is bounded above by a polynomial function of the size of the input.
\item A decision problem lies in $\mathrm{EXP}$, or runs in \emph{exponential time}, if there is a constant $c > 0$ and an algorithm to solve the problem with running time that is bounded above by $2^{n^c}$, where $n$ is the size of the input.
\item A decision problem lies in $\mathrm{E}$ if there is a constant $c >1$ and an algorithm to solve the problem with running time that is bounded above by $c^n$, where $n$ is the size of the input.
\end{enumerate}
\end{definition}

Out of $\mathrm{EXP}$ and $\mathrm{E}$, the latter seems to be more natural at first sight. However, it is less commonly used, for good reason. Typically, we allow ourselves to change the way that the input data is encoded. Alternatively, we may wish to use the solution to one decision problem as a tool for solving another. This may increase (or decrease) the size of the input data by some polynomial function. We therefore would prefer to use complexity classes that are unchanged when $n$ is replaced by a polynomial function of $n$. Obviously $\mathrm{EXP}$ has this nice property whereas $\mathrm{E}$ does not. This is more than just a theoretical issue. For example, there is an algorithm to compute the HOMFLY-PT polynomial of a knot with $n$ crossings in time at most $k^{(\log n) \sqrt{n}}$ for some constant $k$ \cite{BurtonHOMFLY}. This function grows more slowly than $c^n$, for any constant $c > 1$. However, it is widely believed (for very good reason \cite{JaegerVertiganWelsh}) that there is no algorithm to do this in sub-exponential time.

\begin{definition}
There is a \emph{polynomial-time reduction} (or \emph{Karp reduction}) from one decision problem $A$ to another decision problem $B$ if there is a polynomial-time algorithm that translates any given input data for problem $A$ into input data for problem $B$. This translation should have the property that the decision problem $A$ has a positive solution if and only if its translation into problem $B$ has a positive solution.
\end{definition}

A major theme in the field of computational complexity is the use of non-deterministic algorithms. The most important type of such algorithm is as follows.

\begin{definition}
A decision problem lies in $\mathrm{NP}$ (\emph{non-deterministic polynomial time}) if there is an algorithm with the following property. The decision problem has a positive answer if and only if there is a certificate, in other words some extra piece of data, such that when the algorithm is run on the given input data and this certificate, it provides the answer `yes'. This is required to run in time that is bounded above by a polynomial function of the size of the initial input data.
\end{definition}

The phrase `non-deterministic' refers to the fact that the algorithm might only complete its task if it is provided with extra information that must be supplied by some unspecified source. One may wonder why non-deterministic algorithms are of any use at all. But there are several reasons to be interested in them. 

First of all, $\mathrm{NP}$ captures the notion of problems where a positive answer can be verified quickly. For example, is a given positive integer $n$ composite? If the answer is `yes', then one can verify the positive answer by giving two integers greater than one and multiplying them together to get $n$. Once one is given these integers, this multiplication can be achieved in polynomial time as a function of the number of digits of $n$. Hence, this problem lies in NP.

Secondly, problems in $\mathrm{NP}$ \emph{can} be solved deterministically, but with a potentially longer running time. This is because $\mathrm{NP}$ problems lie in $\mathrm{EXP}$, for the following reason. If a problem lies in $\mathrm{NP}$, then there is an algorithm to verify a certificate that runs in time $n^c$, where $n$ is the size of the input data and $c$ is some constant. Since each step of the non-deterministic algorithm can only move the tape of the Turing machine at most one place, the algorithm can only read at most the first $n^c$ digits in the certificate. Thus, we could discard any part of the certificate beyond this without affecting its verifiability. Therefore, we may assume that the certificate has size at most $n^c$. There are only $2^{n^c+1}$ possible strings of 0's and 1's of this length or less. Thus, a deterministic algorithm proceeds by running through all these strings and seeing whether any of them is a certificate that can be successfully verified. If one of these strings is such a certificate, then the algorithm terminates with a `yes'; otherwise it terminates with a `no'. Clearly, this algorithm runs in exponential time. A concrete example highlighting that $\mathrm{NP} \subseteq \mathrm{EXP}$ is again the question of whether a given positive integer $n$ is composite. A certificate for being composite is two smaller positive integers that multiply together to give $n$. So a (rather inefficient) deterministic algorithm simply runs through all possible pairs of integers between $2$ and $n-1$, multiplies them together and checks whether the answer is $n$.

Thirdly, there is the following notion, which is an extremely useful one.

\begin{definition}
A decision problem is $\mathrm{NP}$\emph{-hard} if there is a polynomial-time reduction from any $\mathrm{NP}$ problem to it. If a problem both is in $\mathrm{NP}$ and is $\mathrm{NP}$-hard, then it is termed $\mathrm{NP}$-\emph{complete}.
\end{definition}

It is surprising how many NP-complete problems there are \cite{GareyJohnson}. The most fundamental of these is SAT. This takes as its input a collection of sentences that involve Boolean variables and the connectives AND, OR and NOT, and it asks whether there is an assignment of TRUE or FALSE to each of the variables that makes each of the sentences true. This is clearly a fundamental and universal problem, and so it is perhaps not so surprising that it is NP-complete. But what is striking is that there are so many other problems, spread throughout mathematics, that are NP-complete. As we will see, these include some natural decision problems in topology.

The following famous conjecture is very widely believed.

\begin{conjecture} 
$\mathrm{P} \not= \mathrm{NP}$. Equivalently, any problem that is $\mathrm{NP}$-complete cannot be solved in polynomial time.
\end{conjecture}

Thus when a problem is NP-complete, this provides strong evidence that this problem is difficult. Indeed, in the field of computational complexity, where there are so many fundamental unsolved conjectures, typically the only way to establish any interesting lower bound on a problem's complexity is to prove it conditionally on some widely believed conjecture.

One of the peculiar features of the definition of $\mathrm{NP}$ is that it treats the status of `yes' and `no' answers quite differently. Of course, one could reverse the roles of `yes' and `no' and so one is led to the following definition.

\begin{definition}
A decision problem is in co-NP if its negation is in NP.
\end{definition}

The following conjecture, like the famous $\mathrm{P} \not= \mathrm{NP}$, is also widely believed.

\begin{conjecture} 
\label{Con:NPCoNP}
$\mathrm{NP} \not= \mathrm{co}\textrm{-}\mathrm{NP}$. Hence, if a problem is $\mathrm{NP}$-complete, it does not lie in $\mathrm{co}\textrm{-}\mathrm{NP}$.
\end{conjecture}

One rationale for this conjecture is simply that problems in NP are those where a positive solution can be easily verified. In practice, verifying a positive solution seems quite different from verifying a negative solution. For example, to check a positive answer to an instance of SAT, one need only plug in the given truth values to the Boolean variables and check whether the sentences are all true. However, to verify a negative answer seems, in general, to require that one try out all possible truth assignments of the variables, which is obviously a much lengthier task. Of course, for some instances of SAT, there may be shortcuts, but there does not seem to be a general method that one can apply to verify a negative answer to SAT in polynomial time.

The second part of the above conjecture is a consequence of the first part. For suppose that there were some NP-complete problem $D$ that lies in co-NP. Since any NP problem $D'$ can be reduced to $D$, we would therefore be able to use a certificate for a negative answer to $D$ to provide a certificate for a negative answer to $D'$. Hence, $D'$ would also lie in co-NP.  As $D'$ was an arbitrary NP problem, this would imply that $\mathrm{NP} \subseteq \mathrm{co}\textrm{-}\mathrm{NP}$. This then implies that $\mathrm{co}\textrm{-}\mathrm{NP} \subseteq \mathrm{NP}$ because of the symmetry in the definitions. Hence $\mathrm{NP} = \mathrm{co}\textrm{-}\mathrm{NP}$, contrary to the first part of the conjecture.

It is worth highlighting the following result, due to Ladner \cite{Ladner}. 

\begin{theorem}
If $\mathrm{P} \not= \mathrm{NP}$, then there are decision problems that are in $\mathrm{NP}$, but that are neither in $\mathrm{P}$ nor $\mathrm{NP}$-complete.
\end{theorem}

A problem that is in NP but that is neither in P nor NP-complete is called NP\emph{-intermediate}. There are no naturally-occurring decision problems that are known to be NP-intermediate. Problems that are in NP $\cap$ co-NP but that are not known to be in P are good candidates for being NP-intermediate. As we shall see, there are several decision problems in 3-manifold theory that are of this form. However, for a given problem, it is extremely challenging to provide good evidence for its intermediate status, as there \emph{might} be a polynomial time algorithm to solve it that has not yet been found.

\section{Some highlights}

\subsection{The homeomorphism problem}

This is the most important decision problem in 3-manifold theory. 

\begin{theorem}
\label{Thm:HomeoProblem}
The problem of deciding whether two compact orientable 3-manifolds are homeomorphic is solvable.
\end{theorem}

The problem of deciding whether two links in the 3-sphere are equivalent is nearly a special case of the above result. One must check whether there is a homeomorphism between the link exteriors taking meridians to meridians. This is also possible, and hence we have the following result \cite[Corollary 6.1.4]{Matveev}.

\begin{theorem}
\label{Thm:LinkEquivalence}
The problem of deciding whether two link diagrams represent equivalent links in the $3$-sphere is solvable.
\end{theorem}

There are now several known methods \cite{ScottShort, KuperbergAlgorithmic} for proving Theorem \ref{Thm:HomeoProblem}, but they all use the solution to the Geometrisation Conjecture due to Perelman \cite{Perelman1, Perelman2, Perelman3}. However, the complexity of the problem is a long way from being understood. The best known upper bound is due to Kuperberg \cite{KuperbergAlgorithmic}, who showed that the running time is at most 
$$ 2^{2^{2^{2^{\rdots^t}}}} $$
where $t$ is the sum of the number of tetrahedra in the given triangulations, and the height of the tower is some universal, but currently unknown, constant. We will review some of the ideas that go into this in Section \ref{Sec:HyperbolicStructures}.

The known lower bounds on the complexity of this problem are also very poor. It was proved by the author \cite{LackenbyConditionallyHard} that the homeomorphism problem for compact orientable 3-manifolds is at least as hard as the problem of deciding whether two finite graphs are isomorphic. In a recent breakthrough by Babai \cite{Babai}, graph isomorphism was shown to be solvable in quasi-polynomial time (that is, in time $2^{({\log n})^c}$ for some constant $c$, where $n$ is the sum of the number of vertices in the two graphs). It is not known whether it is solvable in polynomial time, but it is believed by many that it is NP-intermediate.

Given the limitations in our understanding of this important problem, it is natural to ask whether there are any decision problems about 3-manifolds for which we can pin down their complexity. Perhaps unsurprisingly, there are very few decision problems in 3-manifold theory that are known to lie in P. But there are some problems that are known to be NP-complete.

\subsection{Some NP-complete problems}

The following striking result was proved by Agol, Hass and Thurston \cite{AgolHassThurston}.

\begin{theorem}
\label{Thm:3ManifoldKnotGenus}
The problem of deciding whether a knot in a compact orientable $3$-manifold bounds a compact orientable surface with genus $g$ is $\mathrm{NP}$-complete.
\end{theorem}

The ideas that go into this are important, and we will devote much of Sections \ref{Sec:AHT} and \ref{Sec:NPHard} to them.

Recently, some new results have been announced, establishing that some other natural topological problems are NP-complete.  De Mesmay, Rieck,
Sedgwick and Tancer \cite{DRST} proved the following.

\begin{theorem} 
The problem of deciding whether a diagram of the unknot can be reduced to the trivial diagram using at most $k$ Reidemeister moves is $\mathrm{NP}$-complete.
\end{theorem}

\subsection{Some possibly intermediate problems}

There are very few algorithms in 3-manifold theory that run in polynomial time. However, there are some decision problems that are likely to be NP-intermediate or possibly in P. Recall from Section \ref{Sec:ComplexityClasses} that if a problem lies in NP and co-NP, then it is very likely \emph{not} to be NP-complete. 

\begin{theorem}
\label{Thm:UnknotNPCoNP}
The problem of recognising the unknot lies in $\mathrm{NP}$ and $\mathrm{co}$-$\mathrm{NP}$.
\end{theorem}

We will discuss this result in Sections \ref{Sec:AHT} and \ref{Sec:ThurstonNorm}. The proof that unknot recognition lies in NP is due to Hass, Lagarias and Pippenger \cite{HassLagariasPippenger}. The fact that unknot recognition lies in co-NP was first proved by Kuperberg \cite{KuperbergKnottedness}, but assuming the Generalised Riemann Hypothesis. It has now been proved unconditionally by the author \cite{LackenbyEfficientCertification} using a method that was outlined by Agol \cite{AgolCoNP}. It is remarkable that the Generalised Riemann Hypothesis should have relevance in this area of 3-manifold theory. In fact, it remains an assumption in the following theorem of Zentner \cite{Zentner}, Schleimer \cite{Schleimer} and Ivanov \cite{Ivanov}, which builds on work of Rubinstein \cite{Rubinstein} and Thompson \cite{Thompson}.

\begin{theorem}
The problem of deciding whether a 3-manifold is the 3-sphere lies in $\mathrm{NP}$ and, assuming the Generalised Riemann Hypothesis, it also lies in $\mathrm{co}$-$\mathrm{NP}$.
\end{theorem}

In Sections \ref{Sec:AlmostNormal} and \ref{Sec:Homomorphisms}, we explain how this result is proved.

\subsection{Some NP-hard problems}

Much of the progress in algorithmic 3-manifold theory has been to show that certain decision problems are solvable and, in many circumstances, an upper bound on their complexity is given. The task of finding lower bounds on their complexity is more difficult in general. However, there are some interesting problems that have been shown to be NP-hard, even though the problems themselves are not known to be algorithmically solvable.

The \emph{unlinking number} of a link is the minimal number of crossing changes that can be made to the link that turn it into the unlink. It is not known to be algorithmically computable. However, the following result was proved independently by De Mesmay, Rieck, Sedgwick and Tancer \cite{DRST} and by Koenig and Tsvietkova \cite{KoenigTsvietkova}.

\begin{theorem}
The problem of deciding whether the unlinking number of a link is some given integer is $\mathrm{NP}$-hard.
\end{theorem}

The first set of the above authors also considered the following decision problem, which also is not known to be solvable.

\begin{theorem}
The problem of deciding whether a link in $\mathbb{R}^3$ bounds a smoothly embedded orientable surface with zero Euler characteristic in $\mathbb{R}^4_+$ is $\mathrm{NP}$-hard.
\end{theorem}

\section{Pachner moves and Reidemeister moves}
\label{Sec:PachnerReidemeister}

Several decision problems, such as the homeomorphism problem for compact 3-manifolds, may be reinterpreted using Pachner moves.

\begin{definition}
A \emph{Pachner move} is the following modification to a triangulation $\mathcal{T}$ of a closed $n$-manifold: remove a non-empty subcomplex of $\mathcal{T}$ that is isomorphic to the union $F$ of some of but not all of the $n$-dimensional faces of an $(n+1)$-simplex $\Delta^{n+1}$, and then insert the remainder of $\partial \Delta^{n+1} \cut F$. (See Figure \ref{Fig:Pachner}.) For a triangulation $\mathcal{T}$ of an $n$-manifold $M$ with boundary, we also allow the following modification: attach onto its boundary an $n$-simplex $\Delta^n$, by identifying a non-empty subcomplex of $\partial M$ with a subcomplex of $\partial \Delta^n$ consisting of a union of some but not all of the $(n-1)$-dimensional faces.
\end{definition}

\begin{figure}
  \includegraphics[width=4in]{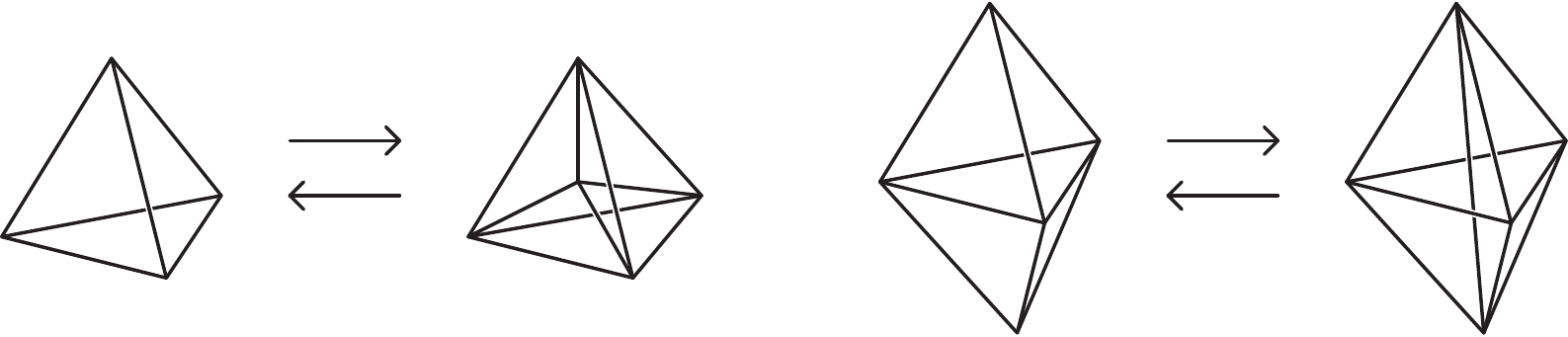}
  \caption{The Pachner moves for a closed $3$-manifold}
  \label{Fig:Pachner}
\end{figure}

The following was proved by Pachner \cite{Pachner}.

\begin{theorem}
\label{Thm:Pachner}
Any two triangulations of a compact PL $n$-dimensional manifold differ by a finite sequence of Pachner moves, followed by a simplicial isomorphism.
\end{theorem}

A \emph{simplicial isomorphism} between simplicial complexes is a simplicial map that is a homeomorphism and hence that has a simplicial inverse. In fact, the final simplicial isomorphism in Theorem \ref{Thm:Pachner} may be replaced by an isotopy, at least when $n =3$. (See the discussion before Theorem 1.1 in \cite{RubinsteinSegermanTillmann} for more details.)

This has an important algorithmic consequence: if one is given two triangulations of compact $n$-dimensional manifolds, and the manifolds are PL-homeomorphic, then one will always be able to prove that they are PL-homeomorphic. This is because one can start with one of the triangulations. One then applies all possible Pachner moves to this triangulation, thereby creating a list of triangulations. Then one applies all possible Pachner moves to each of these, and so on. By Pachner's theorem, the second triangulation will eventually be formed, and hence this gives a proof that the manifolds are PL-homeomorphic. 

Of course, this does not give an algorithm to determine whether two manifolds are (PL-)homeomorphic, because if we are given two triangulations of distinct manifolds, the above procedure does not terminate. However, if one knew in advance how many moves are required, then one would know when to stop. Thus, \emph{if} there were a computable upper bound on the number of moves that are required, then we would have a solution to the PL-homeomorphism problem for compact $n$-manifolds. In fact, in dimension three, the existence of such a bound is \emph{equivalent} to the fact that the homeomorphism problem is solvable. Hence, as a consequence of Theorem \ref{Thm:HomeoProblem}, we have the following.

\begin{theorem}
There is a computable function $P \colon \mathbb{N} \times \mathbb{N} \rightarrow \mathbb{N}$ such that if $\mathcal{T}_1$ and $\mathcal{T}_2$ are triangulations of a compact orientable 3-manifold, then they differ by a sequence of at most $P(|\mathcal{T}_1|, |\mathcal{T}_2|)$ Pachner moves followed by a simplicial isomorphism.
\end{theorem}

\begin{proof} We need to give an algorithm to compute $P(n_1, n_2)$ for positive integers $n_1$ and $n_2$. To do this, we construct all simplicial complexes that are obtained from $n_1$ tetrahedra by identifying some of their faces in pairs. We discard all the spaces that are not manifolds, which is possible since one can detect whether the link of each vertex is a 2-sphere or 2-disc. We also discard all the manifolds that are not orientable. We do the same for simplicial complexes obtained from $n_2$ tetrahedra. Then we use the solution to the homeomorphism problem for compact orientable 3-manifolds to determine which of these manifolds are homeomorphic. Then for each pair of triangulations in our collection that represent the same manifold, we start to search for sequences of Pachner moves relating them. By Pachner's theorem, such a sequence will eventually be found. Thus, $P(n_1, n_2)$ is computable.
\end{proof}

Essentially the same argument gives the following result, using Theorem \ref{Thm:LinkEquivalence}.

\begin{theorem}
There is a computable function $R \colon \mathbb{N} \times \mathbb{N} \rightarrow \mathbb{N}$ such that if $D_1$ and $D_2$ are connected diagrams of a link with $c_1$ and $c_2$ crossings, then they differ by a sequence of at most $R(c_1, c_2)$ Reidemeister moves.
\end{theorem}

Although the functions $P$ and $R$ are computable, it would be interesting to have explicit upper bounds on the number of moves. This is useful even for specific manifolds, such as the 3-sphere, or for specific knots such as the unknot. The smaller the bound one has, the more efficient the resulting algorithm is. In some cases, a polynomial bound can be established, for example, in the following result of the author \cite{LackenbyPolyUnknot}.

\begin{theorem}
\label{Thm:PolyUnknot}
Any diagram for the unknot with $c$ crossings can be converted to the diagram with no crossings using at most $(236 c)^{11}$ Reidemeister moves.
\end{theorem}

This result provides an alternative proof that unknot recognition is NP (one half of Theorem \ref{Thm:UnknotNPCoNP}), which was first proved by Hass, Lagarias and Pippenger \cite{HassLagariasPippenger}. The certificate is simple: just a sequence of Reidemeister moves with length at most $(236 c)^{11}$ taking the given diagram with $c$ crossings to the trivial diagram.

In recent work of the author \cite{LackenbyAllKnotTypes}, this has been generalised to every knot type.

\begin{theorem}
\label{Thm:PolyEveryKnot}
Let $K$ be any link in the 3-sphere. Then there is a polynomial $p_K$ with the following property. Any two diagrams $D_1$ and $D_2$ for $K$ with $c_1$ and $c_2$ crossings can be related by a sequence of at most $p_K(c_1) + p_K(c_2)$ Reidemeister moves.
\end{theorem}

Hence, we have the following corollary.

\begin{corollary}
\label{Cor:KnotTypeRecognitionNP}
For each knot type $K$, the problem of deciding whether a given knot diagram is of type $K$ lies in $\mathrm{NP}$.
\end{corollary}

However, if the knot type is allowed to vary, then the best known explicit upper bound on Reidemeister moves is vast. This is a result of Coward and the author \cite{CowardLackenby}.

\begin{theorem}
\label{Thm:UpperBoundRM}
If $D_1$ and $D_2$ are connected diagrams of the same link, with $c_1$ and $c_2$ crossings, then they are related by a sequence of at most
$$ 2^{2^{2^{2^{\rdots^{c_1+c_2}}}}} $$
Reidemeister moves, where the height of the tower of exponentials is $k^{c_1 + c_2}$. Here, $k = 10^{1000000}$.
\end{theorem}

This was proved using work of Mijatovi\'c \cite{MijatovicKnot}, who provided upper bounds on the number of Pachner moves for triangulations of many 3-manifolds. For the 3-sphere, he obtained the following bound \cite{Mijatovic3Sphere} (see also King \cite{King}).

\begin{theorem}
\label{Thm:PachnerS3}
Any triangulation $\mathcal{T}$ of the 3-sphere may be converted to the standard triangulation, which is the double of a 3-simplex, using at most
$$6 \cdot 10^6 t^2 2^{20000 \, t^2}$$
Pachner moves, where $t$ is the number of tetrahedra of $\mathcal{T}$.
\end{theorem}

This was proved using the machinery that Rubinstein \cite{Rubinstein} and Thompson \cite{Thompson} developed for recognising the 3-sphere. We will discuss this in Section \ref{Sec:AlmostNormal}.

Mijatovi\'c then went on to analyse most Seifert fibre spaces \cite{MijatovicSeifertFibred} and then Haken $3$-manifolds \cite{MijatovicFibreFree} satisfying the following condition. (For simplicity of exposition, we focus on manifolds that are closed or have toral boundary in this definition.)

\begin{definition}
\label{Def:FibreFree}
A compact orientable 3-manifold with (possibly empty) toral boundary is \emph{fibre-free} if when an open regular neighbourhood of its JSJ tori is removed, no component of the resulting 3-manifold fibres over the circle or is the union of two twisted $I$-bundles glued along their horizontal boundary, unless that component is Seifert fibred.
\end{definition}

\begin{theorem}
\label{Thm:PachnerFibreFree}
Let $M$ be a fibre-free Haken 3-manifold with (possibly empty) toral boundary. 
Let $\mathcal{T}_1$ and $\mathcal{T}_2$ be triangulations of $M$, with $t_1$ and $t_2$ tetrahedra. Then they differ by a sequence of at most
$$ 2^{2^{2^{2^{\rdots^{t_1}}}}} +  2^{2^{2^{2^{\rdots^{t_2}}}}} $$
Pachner moves, where the heights of the towers are $c^{t_1}$ and $c^{t_2}$ respectively, possibly followed by a simplicial isomorphism. Here, $c = 2^{200}$.
\end{theorem}

Manifolds that are not fibre-free were also excluded by Haken \cite{HakenHomeomorphism} in his solution to the homeomorphism problem. However, Mijatovi\'c was able to remove the fibre-free hypothesis in the case of knot and link exteriors \cite{MijatovicKnot}, and was thereby able to prove the following result.

\begin{theorem}
\label{Thm:PachnerKnot}
Let $\mathcal{T}_1$ and $\mathcal{T}_2$ be triangulations of the exterior of a knot in the 3-sphere, with $t_1$ and $t_2$ tetrahedra. Then there is a sequence of Pachner moves, followed by a simplicial isomorphism, taking $\mathcal{T}_1$ to $\mathcal{T}_2$ with length at most the bound given in Theorem \ref{Thm:PachnerFibreFree}.
\end{theorem}

This was the main input into the proof of Theorem \ref{Thm:UpperBoundRM}. However, going from a bound on Pachner moves to a bound on Reidemeister moves was not a straightforward task.

The bounds on Pachner and Reidemeister moves presented in this section are an attractive measure of the complexity of the homeomorphism problem for $3$-manifolds and the recognition problem for certain links and manifolds. However, it is worth emphasising that even good bounds on Reidemeister and Pachner moves cannot lead to really efficient algorithms. For example, the polynomial upper bound on Reidemeister moves for the unknot given in Theorem \ref{Thm:PolyUnknot} only establishes that unknot recognition is in $\mathrm{NP}$ and $\mathrm{EXP}$. This is because, without further information, a blind search through polynomially many Reidemeister moves could not do any better than exponential time. Therefore, if we are to find any algorithms in $3$-manifold theory and knot theory that run in sub-exponential time, other methods will be required.

\section{Normal surfaces}
\label{Sec:NormalSurfaces}

Many, but not all, algorithms in 3-manifold theory rely on normal surface theory. Normal surfaces were introduced by Kneser \cite{Kneser} and then were developed extensively by Haken \cite{HakenNormal, HakenHomeomorphism} and many others. In the next three sections, we will give an overview of their theory.

\begin{definition}
An arc properly embedded in a 2-simplex is \emph{normal} if its endpoints are in the interior of distinct edges.
\end{definition}

\begin{definition}
A disc properly embedded in a tetrahedron is a \emph{triangle} if its boundary is three normal arcs. It is a \emph{square} if its boundary is four normal arcs. A \emph{normal disc} is either a triangle or a square.
\end{definition}

\begin{figure}
  \includegraphics[width=3.5in]{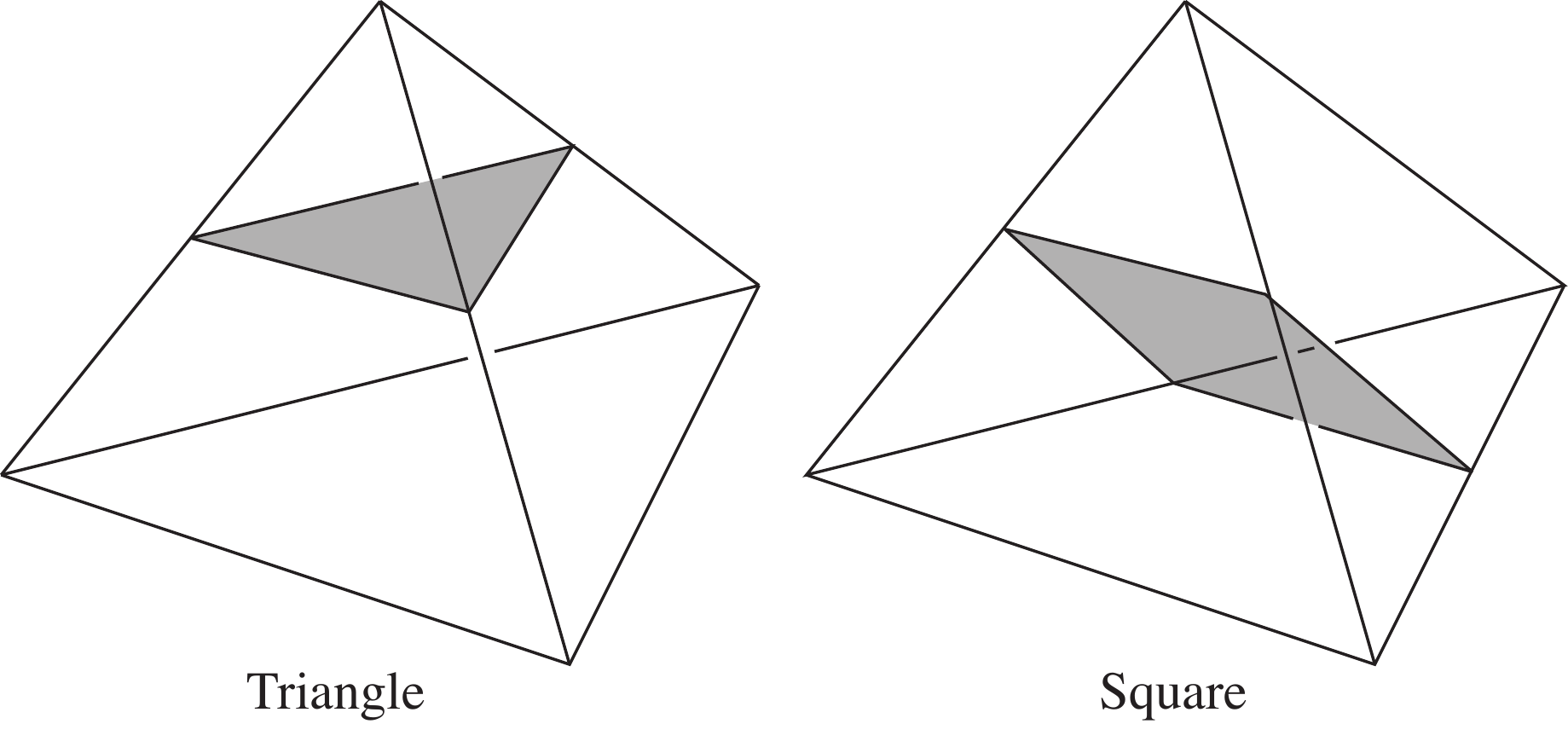}
  \caption{Normal discs}
  \label{Fig:TrianglesSquares}
\end{figure}

\begin{definition}
Let $M$ be a compact 3-manifold with a triangulation $\mathcal{T}$. A surface properly embedded in $M$ is \emph{normal} if its intersection with each tetrahedron of $\mathcal{T}$ is a union of disjoint normal discs.
\end{definition}

\begin{definition} A \emph{normal isotopy} of a triangulated $3$-manifold is an isotopy that preserves each simplex throughout.
\end{definition}

Most interesting surfaces in a 3-manifold can be placed either into normal form or some variant of normal form. For example, we have the following results (see \cite[Proposition 3.3.24, Corollary 3.3.25]{Matveev}).

\begin{theorem}
\label{Thm:NormalDisc}
Let $M$ be a compact orientable 3-manifold that has compressible boundary. Let $\mathcal{T}$ be a triangulation of $M$. Then some compression disc for $\partial M$ is in normal form with respect to $\mathcal{T}$.
\end{theorem}

\begin{theorem}
\label{Thm:EssentialNormal}
Let $M$ be a compact orientable 3-manifold that is irreducible and has incompressible boundary. Let $S$ be a surface properly embedded in $M$ that is \break incompressible and boundary-incompressible, and is neither a sphere nor a boundary-parallel disc. Then $S$ may be isotoped into normal form.
\end{theorem}

The idea behind the proof of these theorems is as follows. In Theorem \ref{Thm:NormalDisc}, let $S$ be a compression disc for $\partial M$. In Theorem \ref{Thm:EssentialNormal}, $S$ is the given surface. First place $S$ in general position with respect to the triangulation $\mathcal{T}$. It then misses the vertices of $\mathcal{T}$ and intersects the edges in a finite collection of points. The number of points is the \emph{weight} of $S$, denoted $w(S)$. This is the primary measure of the complexity of $S$. Each modification that will be made to $S$ will not increase its weight, and many modifications will reduce it. In fact, the weight of $S$ is the most significant quantity in a finite list of other measures of complexity. Each modification will reduce some quantity in this list and will not increase the more significant measures of complexity. Thus, eventually the modifications must terminate, at which stage it can be deduced that the resulting surface is normal.

In outline, the normalisation procedure is as follows. See \cite[Section 3.3]{Matveev} for a more thorough treatment.
\begin{enumerate}
\item Suppose that in some tetrahedron $\Delta$, $S \cap \Delta$ is not a collection of discs. Then $S \cap \Delta$ admits a compression disc $D$ in the interior of $\Delta$.
\item Since $S$ is incompressible, $\partial D$ bounds a disc $D'$ in $S$.
\item When $M$ is irreducible, $D \cup D'$ bounds a ball in $M$, and there is an isotopy that moves $D'$ to $D$.
\item Even when $M$ is reducible, we may remove $D'$ from $S$ and replace it by $D$.
\item This process reduces the measures of complexity and so at some point we must reach a stage where the intersection between $S$ and each tetrahedron is a collection of discs.
\item Suppose that one of these discs intersects an edge of a tetrahedron $\Delta$ more than once. If the interior of the edge lies in the interior of $M$, then there is an isotopy that can be performed that reduces the weight of the surface. This moves $S$ along an \emph{edge compression disc}, which is a disc $E$ in $\Delta$ such that $E \cap \partial \Delta$ is both an arc in $\partial E$ and a sub-arc of an edge of $\Delta$, and where $E \cap S$ is the remainder of $\partial E$.
\item If the above edge lies in $\partial M$, then the disc $E$ forms a potential boundary-compression disc for $S$. However, $S$ is boundary-incompressible and so we may replace a sub-disc of $S$ by $E$. This reduces the weight of $S$.
\item Therefore eventually we reach the stage where $S$ intersects each tetrahedron in a collection of discs and each of these discs intersects each edge of the tetrahedron at most once. It is then normal.
\end{enumerate}

We will see in Section \ref{Sec:AlmostNormal} that other surfaces, particularly certain Heegaard surfaces, may be placed into a variation of normal form, called almost normal form, and that this has some important algorithmic consequences.

\section{The matching equations and fundamental surfaces}
\label{Sec:MatchingFundamental}

One of Haken's key insights was to encode a normal surface in a triangulation $\mathcal{T}$ by counting its number of triangles and squares of each type in each tetrahedron. 

\begin{definition}
The \emph{vector} $(S)$ associated with a normal surface $S$ is the $(7|\mathcal{T}|)$-tuple of non-negative integers that counts the number of triangles and squares of each type in each tetrahedron.
\end{definition}

The vector of a properly embedded normal surface $S$ satisfies some fairly obvious conditions:
\begin{enumerate}
\item The co-ordinates have to be non-negative integers. 
\item Any two squares of different types within a tetrahedron necessarily intersect, and so this imposes constraints on $(S)$. These assert that for each pair of distinct square types within a tetrahedron, at least one of the corresponding co-ordinates of $(S)$ is zero. These are called the \emph{compatibility conditions} (or the \emph{quadrilateral conditions}).
\item For each face $F$ of $\mathcal{T}$ with tetrahedra on both sides, the intersection $S \cap F$ consists of normal arcs. These come in three different types. The number of arcs of each type can be computed from the number of the triangles and squares of particular types in one of the adjacent tetrahedra. Similarly, the number of arcs of this type can be computed from the number of triangles and squares in the other adjacent tetrahedron. Thus these numbers of triangles and squares satisfy a linear equation. There is one such equation for each normal arc type in each face with tetrahedra on both sides. These are called the \emph{matching equations}. (See Figure \ref{Fig:MatchingEquations}).
\end{enumerate}

\begin{figure}
  \includegraphics[width=3in]{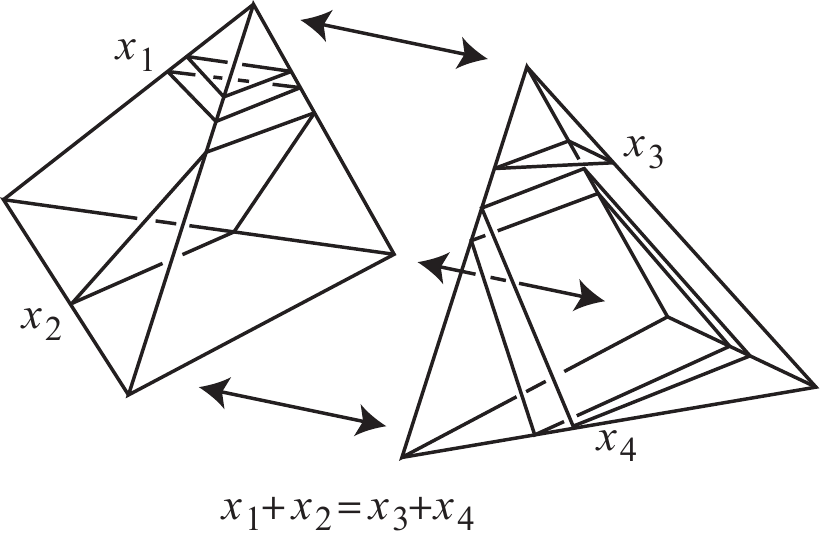}
  \caption{A matching equation}
  \label{Fig:MatchingEquations}
\end{figure}

The following key observation is due to Haken \cite[Theorem 3.3.27]{Matveev}.

\begin{lemma}
\label{Lem:OneOneSurfaceVector}
There is a one-one correspondence between properly embedded normal surfaces up to normal isotopy and vectors in $\mathbb{R}^{7|\mathcal{T}|}$ satisfying the above three conditions.
\end{lemma}

It is therefore natural to try to understand and exploit this structure on the set of normal surfaces. One is quickly led to the following definition.

\begin{definition}
A properly embedded normal surface $S$ is the \emph{normal sum} of two properly embedded normal surfaces $S_1$ and $S_2$ if $(S) = (S_1) + (S_2)$. We write $S = S_1 + S_2$.
\end{definition}

This has an important topological interpretation. We can place $S_1$ and $S_2$ in general position, by performing a normal isotopy to one of them. Then they intersect in a collection of simple closed curves and properly embedded arcs. It turns out that $S_1 + S_2$ can be obtained from $S_1 \cup S_2$ by cutting along these arcs and then regluing the surfaces in a different way. For example, consider a simple closed curve of $S_1 \cap S_2$, and suppose that this is curve is orientation-preserving in both $S_1$ and $S_2$. Then one removes from $S_1$ and $S_2$ annular regular neighbourhoods of this curve, and then reattaches disjoint annuli to form $S_1 + S_2$. In each case, we remove a subsurface from $S_1$ and $S_2$ (consisting of annuli, M\"obius bands and discs) and then we reattach a surface with the same Euler characteristic. So we obtain the following consequence.

\begin{lemma}
The normal sum $S_1 + S_2$ satisfies $\chi(S_1 + S_2) = \chi(S_1) + \chi(S_2)$.
\end{lemma}

There is another way of proving this lemma, by observing that the Euler characteristic of a normal surface $S$ is a linear function of its vector $(S)$. This is because its decomposition into squares and triangles gives a cell structure on $S$, and the numbers of vertices, edges and faces of this cell structure are linear functions of the co-ordinates of $(S)$. This alternative proof has an important consequence: one can compute $\chi(S)$ in polynomial time as a function of the number of digits of the co-ordinates of $(S)$.


The following is also immediate, because $S_1 + S_2$ and $S_1 \cup S_2$ intersect the 1-skeleton of the triangulation at exactly the same points.

\begin{lemma}
\label{Lem:WeightAdditive}
The normal sum $S_1 + S_2$ satisfies $w(S_1 + S_2) = w(S_1) + w(S_2)$.
\end{lemma}

By a careful analysis of normal summation, the following result can be obtained. See Matveev \cite[Theorems 4.1.13 and 4.1.36 and the proof of Theorem 6.3.21]{Matveev}.

\begin{theorem}
\label{Thm:EssentialSummand}
Let $M$ be a compact orientable irreducible 3-manifold with a triangulation. 
\begin{enumerate}
\item
Suppose that $M$ has compressible boundary, and let $D$ be a compression disc for $\partial M$ that is normal and that has least weight in its isotopy class. Then if $D$ is a normal sum $S_1 + S_2$, then neither $S_1$ nor $S_2$ can be a sphere or a boundary-parallel disc.
\item
Suppose that $M$ has (possibly empty) incompressible boundary. Let $S$ be a connected incompressible boundary-incompressible surface properly embedded in $M$ that is not a sphere, a disc, a projective plane  or a boundary-parallel torus. Suppose that $S$ is normal and that it has least weight in its isotopy class. Then if $S$ is a normal sum $S_1 + S_2$, then $S_1$ and $S_2$ are incompressible and boundary-incompressible, and neither is a sphere, disc, projective plane or boundary-parallel torus.
\end{enumerate}
\end{theorem}

\begin{definition}
A normal surface is \emph{fundamental} if it cannot be written as a sum of non-empty normal surfaces.
\end{definition}

Clearly, any normal surface can be expressed as a sum of fundamental surfaces. The following results demonstrate the importance of fundamental surfaces. Part (1) of the theorem is due to Haken \cite{HakenNormal}; part (2) is due to Jaco and Oertel \cite[Theorem 2.2]{JacoOertel} (see also \cite[Theorems 4.1.13 and 4.1.30]{Matveev}).

\begin{theorem} \label{Thm:Fundamental}
Let $M$ be a compact orientable irreducible 3-manifold with a triangulation.
\begin{enumerate}
\item If $M$ has compressible boundary, then there is a compressing disc that is normal and fundamental.
\item If $M$ is closed and contains a properly embedded orientable incompressible surface other than a sphere, then it contains one that is the boundary of a regular neighbourhood of a fundamental surface.
\end{enumerate}
\end{theorem}

\begin{proof}
We focus on (1), as the proof of (2) is similar. Let $D$ be a compression disc that is normal and has least possible weight. Suppose that $D$ is a normal sum $S_1 + S_2$. Since $1 = \chi(D) = \chi(S_1) + \chi(S_2)$, we deduce that some $S_i$ has positive Euler characteristic. By focusing on one of its components, we may assume that $S_i$ is connected. It cannot be a sphere by Theorem \ref{Thm:EssentialSummand} (1). It cannot be a projective plane, since the only irreducible orientable 3-manifold containing a projective plane is $\mathbb{RP}^3$, which is closed. Hence, it must be a disc. This is not boundary parallel, by Theorem \ref{Thm:EssentialSummand} (1). Hence, it is also a compression disc. But by Lemma \ref{Lem:WeightAdditive}, it has smaller weight than $D$, which is a contradiction. Hence, $D$ must have been fundamental.
\end{proof}

We will give a proof of the following result in the next section.

\begin{theorem} \label{Thm:ConstructFundamental}
A triangulation of a compact 3-manifold supports only finitely many fundamental normal surfaces and there is an algorithm to list them all.
\end{theorem}

As a consequence of Theorems \ref{Thm:Fundamental} (1) and \ref{Thm:ConstructFundamental}, we get the following famous result of Haken \cite{HakenNormal}.

\begin{theorem}
\label{Thm:AlgorithmCompressible}
There is an algorithm to decide whether a compact orientable 3-manifold has compressible boundary. Hence, there is an algorithm to decide whether a given knot is the unknot.
\end{theorem}

\begin{proof} 
The input to the first algorithm is a triangulation of the 3-manifold $M$. The input to the second algorithm is either a triangulation of the knot exterior $M$ or a diagram of the knot, but in the latter case, the first step in the algorithm is to use the diagram to create a triangulation. Using Theorem \ref{Thm:ConstructFundamental}, one simply lists all the fundamental surfaces in $M$. For each one, the algorithm determines whether it is a disc. If it is, then the algorithm determines whether its boundary is an essential curve in $\partial M$. If there is such a disc, then $\partial M$ is compressible. If there is not, then by Theorem \ref{Thm:Fundamental} (1), $\partial M$ is incompressible.
\end{proof}

The following is an important extension of this result \cite[Theorem 6.3.17]{Matveev}.

\begin{definition}
A compact orientable 3-manifold is \emph{simple} if it is irreducible and has incompressible boundary and it contains no properly embedded essential annuli or tori.
\end{definition}

\begin{theorem}
\label{Thm:ListSurfacesForSimpleManifold}
There is an algorithm that takes, as its input, a triangulation of a compact orientable simple 3-manifold $M$ and an integer $k$, and it provides a list of all connected orientable incompressible boundary-incompressible properly embedded surfaces in $M$ with Euler characteristic at least $k$, up to ambient isotopy.
\end{theorem}

In the above theorem, there is no requirement that different surfaces in the list are not isotopic. However, it is possible to arrange this with more work, using a result of Waldhausen \cite[Proposition 5.4]{Waldhausen} which controls the way that isotopic surfaces intersect each other.

\begin{proof}
Let $S_1, \dots, S_n$ be the fundamental normal surfaces in the given triangulation. Let $S$ be a connected, incompressible, boundary-incompressible surface with $\chi(S) \geq k$, other than a sphere, a disc or a boundary-parallel torus. By Theorem \ref{Thm:EssentialNormal}, it can be isotoped into normal form. Pick a least weight representative for it, also called $S$. Then $S$ is a normal sum $\lambda_1 S_1 + \dots + \lambda_n S_n$, where each $\lambda_i$ is a non-negative integer. By Theorem \ref{Thm:EssentialSummand} (2), any $S_i$ that is a sphere, disc or boundary-parallel torus occurs with $\lambda_i = 0$. The same is true for any $S_i$ that is compressible or boundary-compressible. By our hypothesis that $M$ is simple, any $S_i$ that is an annulus or torus therefore has $\lambda_i =0$. No $S_i$ can be a projective plane, as the only irreducible orientable 3-manifold containing a projective plane is $\mathbb{RP}^3$, and this contains no properly embedded orientable incompressible surfaces. It might be the case that some $S_i$ is a Klein bottle or M\"obius band, but in this case, $\lambda_i \leq 1$, as otherwise $2 S_i$ is summand and this is a torus or annulus. The remaining surfaces all have Euler characteristic at most $-1$. Hence, the sum of the coefficients $\lambda_i$ for these surfaces is at most $-k$. Therefore, also taking account of possible M\"obius bands and Klein bottles, we deduce that $\sum_i \lambda_i \leq -k + n$. 

Hence, there are only finitely many such surfaces and they may all be listed. If we wish only to list those that are actually incompressible and boundary-incompressible, then we can cut along each surface and verify whether it is incompressible and boundary-incompressible, using a variation of Theorem \ref{Thm:AlgorithmCompressible}. This actually requires the use of boundary patterns, which are discussed in Section \ref{Sec:Hierarchies}.
\end{proof}

One might wonder whether it is necessary to assume that the manifold $M$ is simple in the above theorem. However, if $M$ contains an essential annulus or torus, say, then it is possible to Dehn twist along such a surface and so the manifold might have infinite mapping class group. If there is a surface $S$ that intersects the annulus or torus non-trivially, then the image of $S$ under powers of this Dehn twist will, in general, form an infinite collection of non-isotopic surfaces all with the same Euler characteristic. Hence, in this case, the conclusion of Theorem \ref{Thm:ListSurfacesForSimpleManifold} does not hold.

For non-simple manifolds, it is natural to consider their canonical tori and annuli. Recall that a torus or annulus $S$ properly embedded in a compact orientable 3-manifold $M$ is \emph{canonical} if it is essential and, given any other essential annulus or torus properly embedded in $M$, there is an ambient isotopy that pulls it off $S$. Canonical annuli and tori are also called \emph{JSJ annuli and tori} due to the work of Jaco, Shalen and Johannson \cite{JacoShalen, Johannson}. The exterior of the canonical tori is the \emph{JSJ decomposition} of $M$. It was shown by Jaco, Shalen and Johannson  that, when $M$ is a compact orientable irreducible 3-manifold with (possibly empty) toral boundary, each component of its JSJ decomposition is either simple or Seifert fibred.

Again by a careful analysis of normal summation, the following was proved by Haken \cite{HakenHomeomorphism} and Mijatovi\'c \cite[Propositions 2.4 and 2.5]{MijatovicFibreFree}. See also Matveev \cite[Theorem 6.4.31]{Matveev}.

\begin{theorem}
\label{Thm:ConstructJSJ}
Let $M$ be a compact orientable irreducible 3-manifold with incompressible boundary and let $\mathcal{T}$ be a triangulation of $M$ with $t$ tetrahedra. Then there is an algorithm to construct the canonical annuli and tori of $M$. In fact, they may be realised as a normal surface with weight at most $2^{81t^2}$.
\end{theorem}

\section{The exponential complexity of normal surfaces}
\label{Sec:Sec:ExponentialNormal}



\subsection{An upper bound on complexity}

As we have seen, the fundamental surfaces form the building blocks for all normal surfaces. The following important result \cite{HassLagariasPippenger} provides an upper bound on their weight.

\begin{theorem}
\label{Thm:WeightFundamental}
The weight of a fundamental normal surface $S$ satisfies $w(S) \leq t^2 2^{7t + 7}$, where $t$ is the number of tetrahedra in the given triangulation $\mathcal{T}$.
\end{theorem}

Note that this immediately implies Theorem \ref{Thm:ConstructFundamental}, as one can easily list all the normal surfaces with a given upper bound on their weight.

The proof of Theorem \ref{Thm:WeightFundamental} relies on the structure of the set of all normal surfaces, which we now discuss.

Define the \emph{normal solution space} $\mathcal{N}$ to be the subset of $\mathbb{R}^{7t}$ consisting of points $v$ that satisfy the following conditions:
\begin{enumerate}
\item each co-ordinate of $v$ is non-negative;
\item $v$ satisfies the compatibility conditions;
\item $v$ satisfies the matching equations.
\end{enumerate}
Thus, the vectors of normal surfaces are precisely $\mathcal{N} \cap \mathbb{Z}^{7t}$ by Lemma \ref{Lem:OneOneSurfaceVector}. Now, $\mathcal{N}$ is a union of convex polytopes, as follows. Each polytope $\mathcal{C}$ is formed by choosing, for each tetrahedron of $\mathcal{T}$, two of its square types, whose co-ordinates are set to zero. Each polytope $\mathcal{C}$ is clearly a convex subset of $\mathbb{R}^{7t}$. Equally clearly, if a vector $v$ lies in $\mathcal{C}$, then so does any positive multiple of $v$. Thus, it is natural to consider the intersection $\mathcal{P}$ between $\mathcal{C}$ and $\{ (x_1, \dots, x_{7t}) : x_1 + \dots + x_{7t} = 1 \}$. Then $\mathcal{C}$ is a cone over $\mathcal{P}$, with cone point the origin. This set $\mathcal{P}$ is clearly compact and convex, and in fact it is a polytope. Its faces are obtained as the intersection between $\mathcal{P}$ and some hyperplanes of the form $\{ x_i = 0 \}$. In particular, each vertex is the unique solution to the following system of equations:
\begin{enumerate}
\item the matching equations;
\item extra equations of the form $x_i = 0$ for certain integers $i$;
\item $x_1 + \dots + x_{7t} = 1$.
\end{enumerate}
Because any such vertex is a unique solution to these equations, it has rational co-ordinates. (We will discuss this further below.) Hence, some multiple of this vertex has integer entries, and therefore corresponds to a normal surface. The smallest non-zero multiple is termed a \emph{vertex} surface.

\begin{proof}[Proof of Theorem \ref{Thm:WeightFundamental}] We first bound the size of the vector $(S)$ of any vertex surface $S$. By definition, this is a multiple of a vertex $v$ of one of the polytopes $\mathcal{P}$ described above. This was the unique solution to the matrix equation $B v = (0, 0, \dots, 0, 1)^T$ for some integer matrix $B$. Note that each row of $B$, except the final one, has at most $4$ non-zero entries and in fact the $\ell^2$ norm of this row is at most $4$. Since the solution is unique, $B$ has maximal rank $7t$, and hence some square sub-matrix $A$, consisting of some subset of the rows of $B$, is invertible. In other words, $v = A^{-1} (0, 0, \dots, 0, 1)^T$. Now, $A^{-1} = \mathrm{det}(A)^{-1} \mathrm{adj}(A)$. Here, $\mathrm{adj}(A)$ is the adjugate matrix, each entry of which is obtained by crossing out some row and some column of $A$ and then taking the determinant of this matrix. So, $\mathrm{det}(A) v = \mathrm{adj}(A) (0, 0, \dots, 0, 1)^T$ is an integral matrix and hence corresponds to an actual normal surface. We can bound the co-ordinates of this normal surface by noting that the determinant of a matrix has modulus at most the product of the $\ell^2$ norms of the rows of the matrix, and so each entry of $ \mathrm{adj}(A)$ has modulus at most $(\sqrt{7t})4^{7t-1}$. Hence, the vector $(S)$, which is the smallest non-zero multiple of $v$ with integer entries, also has this bound on its co-ordinates.

Now consider a fundamental surface $S$. It lies in one of the subsets $\mathcal{C}$ described above that is a cone on a polytope $\mathcal{P}$. Since $\mathcal{P}$ is the convex hull of its vertices $v_i$, we deduce that every element of $\mathcal{P}$ is of the form $\sum_i \lambda_i v_i$, where each $\lambda_i \geq 0$ and $\sum_i \lambda_i = 1$. In fact, we may assume that all but at most $7t$ of these $\lambda_i$ are zero, because if more than $7t$ were non-zero, we could use the linear dependence between the corresponding $v_i$ to reduce one of the $\lambda_i$ to zero. We deduce that every element of $\mathcal{C}$ is of the form $\sum_i \mu_i (S_i)$, where each $S_i$ is a vertex surface, each $\mu_i \geq 0$ and at most $7t$ of the $\mu_i$ are non-zero.  Clearly, if $S$ is fundamental, then each $\mu_i < 1$, as otherwise $(S)$ is the sum $(S_i) + ((S) - (S_i))$ and hence is not fundamental. So, each co-ordinate of $S$ is at most $(7t)^{3/2} 4^{7t-1}$. In fact, a slightly more refined analysis gives the bound $7t 2^{7t-1}$ (see the proof of \cite[Lemma 6.1]{HassLagariasPippenger}).

This gives a bound on the weight of $S$. Each co-ordinate correponds to a triangle or square of $S$ and hence contributes at most $4$ to the weight of $S$. Therefore, $w(S)$ is at most $4$ times the sum of its co-ordinates. This is at most $(49 t^2) 2^{7t+1}$, which is at most the required bound.
\end{proof}

\subsection{A lower bound on complexity}

Theorem \ref{Thm:WeightFundamental} gives an explicit upper bound on the number of triangles and squares in a fundamental normal surface. This is essentially an exponential function of $t$, the number of tetrahedra in the triangulation. One might wonder whether we can improve this to a sub-exponential bound, perhaps even a polynomial one. However, examples due to Hass, Snoeyink and Thurston \cite{HassSnoeyinkThurston} demonstrate that this is not possible. We present them in this subsection.

\begin{figure}
  \includegraphics[width=4in]{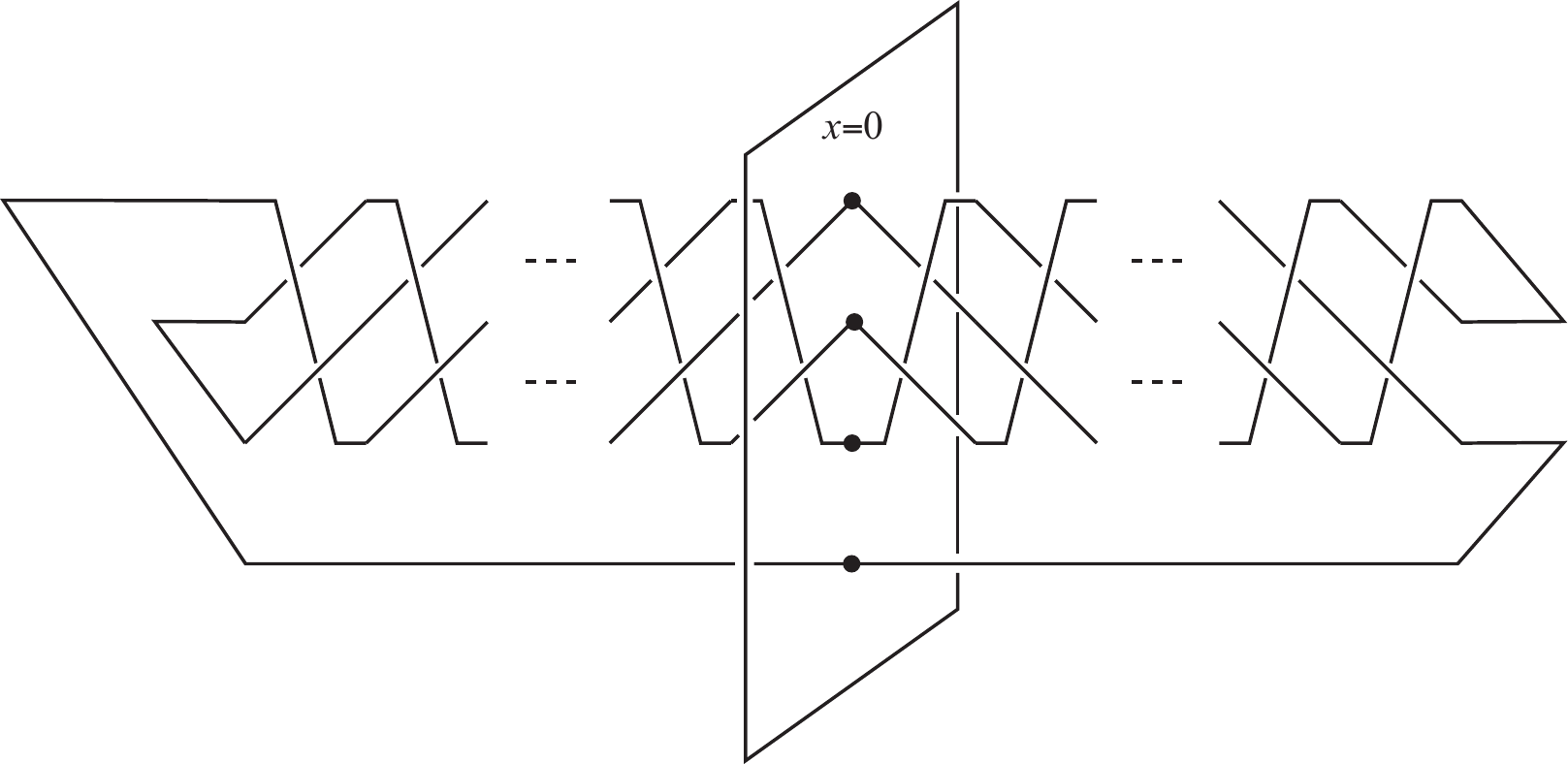}
  \caption{The PL unknot $K_n$}
  \label{Fig:HassSnoeyinkThurston}
\end{figure}

These examples are polygonal curves $K_n$ in $\mathbb{R}^3$, one for each natural number $n$. (See Figure \ref{Fig:HassSnoeyinkThurston}.) Each is composed of $10n+9$ straight edges. It is arranged much like one of the standard configurations of a 2-bridge knot. Thus, the majority of the curve lies in $\mathbb{R}^3$ like a 4-string braid, in the sense that it is transverse to the planes $\{ x = \textrm{constant} \}$. Here, we view the plane of the diagram as the $x-y$ plane, with the $z$-axis pointing vertically towards the reader. This braid is of the form $(\sigma_1 \sigma_2^{-1})^n (\sigma_2 \sigma_1^{-1})^n$, where $\sigma_1$ and $\sigma_2$ are the first two standard generators for the 4-string braid group. The left and right of the braid are capped off to close $K_n$ into a simple closed curve. Since the braid $(\sigma_1 \sigma_2^{-1})^n (\sigma_2 \sigma_1^{-1})^n$ is trivial, all these knots are topologically the same. In fact, it is easy to see that they are unknotted, and hence they bound a disc, which we may assume is a union of triangles that are flat in $\mathbb{R}^3$. The following result gives a lower bound on the complexity of such a disc.

\begin{theorem} 
\label{Thm:HassSnoeyinkThurston}
Any piecewise linear disc bounded by $K_n$ consists of at least $2^{n-1}$ flat triangles.
\end{theorem}

We give an outline of the proof and refer the reader to \cite{HassSnoeyinkThurston} for more details. 

The first step is to exhibit a specific disc $D_n$ that $K_n$ bounds which is smooth in its interior. In the case of $K_0$, the disc $D_0$ is shown in Figure \ref{Fig:HassSnoeyinkThurston2}. It lies in the plane $\{ z = 0 \}$. Note that the intersection between $D_0$ and the plane $\{ x = 0 \}$ consists of two straight curves, which we denote by $\beta$.

\begin{figure}
  \includegraphics[width=3in]{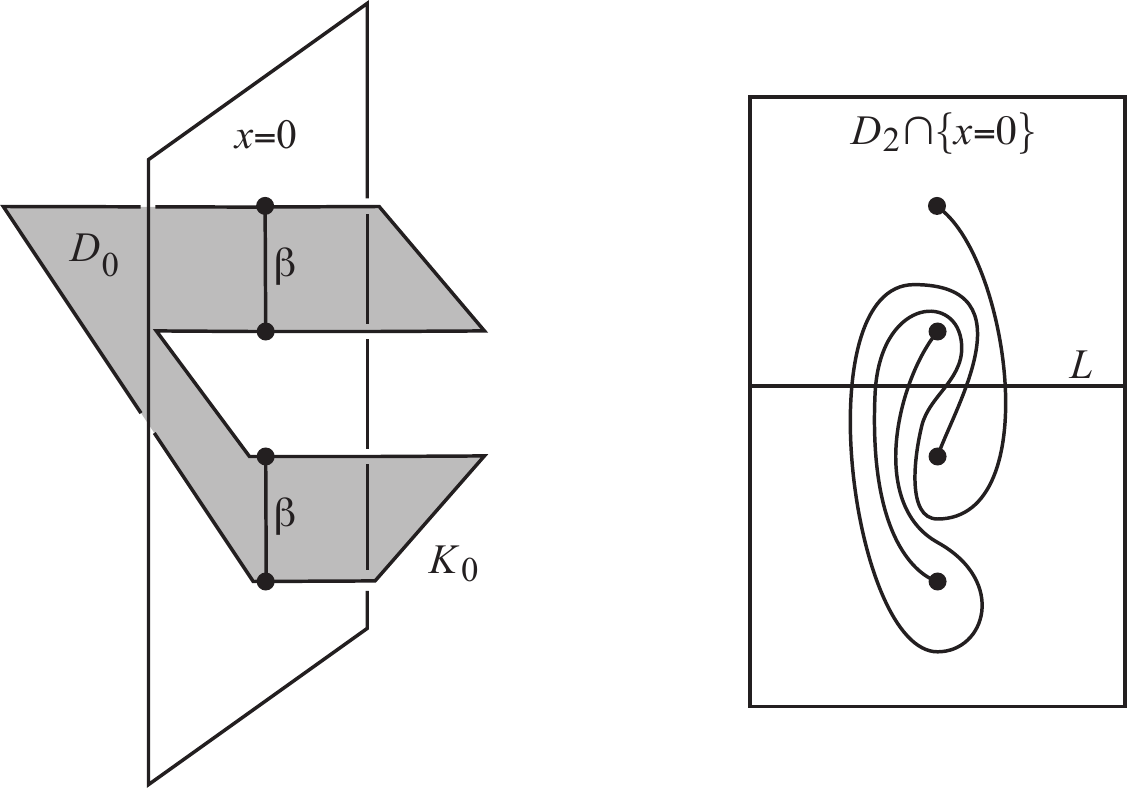}
  \caption{Left: $K_0$ and its spanning disc $D_0$; Right: $D_2 \cap \{ x = 0 \}$}
  \label{Fig:HassSnoeyinkThurston2}
\end{figure}

To obtain $K_n$ from $K_{n-1}$, the following homeomorphism is applied to $\mathbb{R}^3$. This preserves each of the planes $\{ x = \mathrm{constant} \}$. It is supported in a small regular neighbourhood of the plane $\{ x= 0 \}$ that separates the braid $(\sigma_1 \sigma_2^{-1})^{n-1}$ from the braid $(\sigma_2 \sigma_1^{-1})^{n-1}$. More specifically, a small positive real number $\epsilon_n$ is chosen and the homeomorphism is supported in $\{ - \epsilon_n  \leq x \leq \epsilon_n \}$.
At $\{ x = - \epsilon_n \} $, the homeomorphism is the identity. As $x$ increases from $-\epsilon_n$, the homeomorphism of the planes realises the braid generator $\sigma_1$, and then the braid generator $\sigma_2^{-1}$. Thus, the homeomorphism applied to the plane $\{ x = 0 \}$ is the map $\phi$ that in the mapping class group of $(\mathbb{R}^2, 4 \ \mathrm{points})$ represents the braid $\sigma_1 \sigma_2^{-1}$. Between $\{ x = 0 \}$ and $\{ x = \epsilon_n \}$, the above homeomorphisms are applied but in reverse order, so that by the time we reach $\{ x = \epsilon_n \}$, the homeomorphism is the identity.
Thus, this specifies a homeomorphism $\mathbb{R}^3 \rightarrow \mathbb{R}^3$ taking $K_{n-1}$ to $K_n$. We define $D_n$ to be the image of $D_{n-1}$ under this homeomorphism.

Now, within the plane $\{ x = 0 \}$, there is a straight line $L$ dividing the top two points of $K_n$ from the bottom two points. The key to understanding the complexity of $D_n$ is to see how many times it intersects this line. The intersection between $D_n$ and $\{ x = 0 \}$ is the image of the two arcs $\beta$ under the homeomorphism $\phi^n$. Now, $\phi$ is a pseudo-anosov homeomorphism with dilatation $\lambda > 1$ (which is in fact the square of the golden ratio). It is a standard result about pseudo-anosovs that as $n \rightarrow \infty$, the number of intersections between $\phi^n(\beta)$ and $L$ grows exponentially. In fact, it grows like at least $\lambda^n$ (see \cite[Theorem 14.24]{FarbMargalit}).

Although $D_n$ is smooth, any piecewise linear approximation to it will have this same property: the number of intersections between it and the line $L$ grows at least like $\lambda^n$. But a straight triangle intersects a straight line in either a line segment or at most one point. Hence, any piecewise-linear approximation to $D_n$ must have at least exponentially many straight triangles.

So far, we have only considered the disc $D_n$ and piecewise-linear discs that approximate it. The main part of the proof of Theorem \ref{Thm:HassSnoeyinkThurston} is to show that \emph{any} disc $E_n$ bounded by $K_n$ must intersect the line $L$ at least $2^{n-1}$ times. Now, $D_n$ has a special property: it is transverse to the planes $\{ x= \mathrm{constant} \}$ containing the braid  $(\sigma_1 \sigma_2^{-1})^n (\sigma_2 \sigma_1^{-1})^n$. Thus if $\{ x = 0 \} $ is the plane in the middle of the braid and $\{ x = c \}$ is the plane at one end, then $D_n \cap \{ x = 0 \}$ and $D_n \cap \{ x = c \}$ are related by $\phi^n$ up to isotopy. Thus, at least one of these intersections must be `exponentially complicated' and in the case of $D_n$, it is the plane $\{ x = 0 \}$ that is complicated.

Now, an arbitrary spanning disc $E_n$ need not be transverse to the planes $\{ x = \mathrm{constant} \}$. The co-ordinate $x$ can be viewed as a Morse function on $E_n$. This may have critical points and at the planes on either side of such a critical point, the isotopy classes of the intersection between $E_n$ and these planes may change. A key part of the argument of Hass, Snoeyink and Thurston establishes that, in fact, there can be only \emph{one} such saddle where the isotopy class changes in any interesting way. In particular, one of the regions $\{ -c \leq x \leq 0 \}$ and $\{ 0 \leq x \leq c\}$ does not contain such a saddle (say the latter) and hence the intersection between $E_n$ and $\{ x = 0 \}$ or $\{ x = c \}$ must be exponentially complicated. In fact, it must be $\{ x = 0 \}$ that has exponentially complicated intersection, but this point is not essential for the proof of their theorem.

We now explain briefly why there is at most one relevant saddle singularity of $E_n$. Suppose that a saddle occurs in the plane $\{ x = k \}$. We may assume that the saddles of $E_n$ occur at different $x$ co-ordinates. Hence, the intersection between $E_n$ and $\{ x = k \}$ is a graph with a single 4-valent vertex in the interior of $E_n$ and possibly some 1-valent vertices on the boundary. If there are $0$ or $2$ vertices on the boundary, then this is not an `interesting' saddle and the intersection between $E_n$ and the planes just to the left and right of $\{ x = k \}$
are basically the same. On the other hand, when there are $4$ vertices on the boundary, then these 4 vertices divide $K_n$ into $4$ arcs, each of which must contain a critical point of $K_n$ with respect to the function $x$. However, $K_n$ has only $4$ critical points. One can deduce that if there was more than one such saddle, then in fact $K_n$ would have to have more than $4$ critical points, which is manifestly not the case. This completes the sketch of Theorem \ref{Thm:HassSnoeyinkThurston}.

Note that this theorem provides a lower bound on the number of triangles and squares for normal discs, as follows. It is straightforward to build a triangulation $\mathcal{T}$ of the exterior of $K_n$, where the number of tetrahedra is a linear function of $n$ and each tetrahedron is straight in $\mathbb{R}^3$. Any normal surface in $\mathcal{T}$ can be realised as a union of flat triangles, possibly by subdividing each square along a diagonal into two triangles. Hence, Theorem \ref{Thm:HassSnoeyinkThurston} provides an exponential lower bound on the number of triangles and squares in any normal spanning disc in $\mathcal{T}$.

\section{The algorithm of Agol-Hass-Thurston}
\label{Sec:AHT}

As we saw in the previous section, the normal surfaces that we are interested in (such as a spanning disc for an unknot) may be exponentially complicated, as a function of the number of tetrahedra in our given triangulation. Clearly, this is problematic if one is trying to construct efficient algorithms. For example, suppose we are given a normal surface via its vector and we want to determine whether it is a disc. If the vector has exponential size, then we could not hope for our algorithm to build the surface in polynomial time in order to determine its topology. One can easily compute its Euler characteristic, since this is a linear function of the co-ordinates of the vector. So one can easily verify whether the Euler characteristic is $1$ and whether the surface has non-empty boundary. But this is not quite enough information to be able to deduce that the surface is a disc: one also needs to know that it is connected. A very useful algorithm, due to Agol-Hass-Thurston (AHT), allows us to verify this, even when the surface is exponentially complicated. The algorithm is very general and so has many other applications. In fact, it is fair to say that, by using the AHT algorithm cleverly, one can answer just about any reasonable question in polynomial time about a normal surface with exponential weight.

\subsection{The set-up for the algorithm}

Initially, we will consider just the problem of determining the number of components of a normal surface. This can be solved using a `vanilla' version of the AHT algorithm. We will then introduce some greater generality and will give some examples where this is useful.

One can think of the problem of counting the components of a normal surface $S$ as the problem of counting certain equivalence classes. Specifically, consider the points of intersection between $S$ and the 1-skeleton $\mathcal{T}^{1}$ of the triangulation $\mathcal{T}$. If two such points are joined by a normal arc of intersection between $S$ and a face of $\mathcal{T}$, then clearly they lie in the same component of $S$. In fact, two points of $S \cap \mathcal{T}^{1}$ are in the same component of $S$ if and only if they are joined by a finite sequence of these normal arcs. Thus, one wants to consider the equivalence relation on $S \cap \mathcal{T}^{1}$ that is generated by the relation `joined by a normal arc of intersection between $S$ and a face of $\mathcal{T}$'.

Now one can think of the edges of $\mathcal{T}$ as arranged disjointly along the real line. Then $S \cap \mathcal{T}^{1}$ is a finite set of points in $\mathbb{R}$, which we can take to be the integers between $1$ and $N$, say, denoted $[1,N]$. (See Figure \ref{Fig:AgolHassThurston1}.) In each face of $\mathcal{T}$, there are at most three types of normal arc, and each such arc type specifies a bijection between two intervals in $[1,N]$. This bijection is very simple: it is either $x \mapsto x + c$ for some integer $c$ or $x \mapsto c - x$ for some integer $c$.

Thus, the `vanilla' version of the AHT algorithm is as follows:

\emph{Input}: (1) a positive integer $N$;
\vspace{-0.05in}
\begin{enumerate}
\item[(2)] a collection of $k$ bijections between sub-intervals of $[1,N]$, called \emph{pairings}; 
each pairing is of the form $x \mapsto x+ c$ (known as \emph{orientation-preserving}) or $x \mapsto c - x$ (known as \emph{orientation-reversing}) for some integer $c$.
\end{enumerate}

\emph{Output}:
the number of equivalence classes for the equivalence relation generated by the pairings.

The running time for the algorithm is a polynomial function of $k$ and $\log N$. It is the $\log$ that is crucial here, because this allows us to tackle exponentially complicated surfaces in polynomial time.

In fact, one might want to know more than just the number of components of $S$. For example, one might be given two specific points of $S \cap \mathcal{T}^{1}$ and want to know whether they lie in the same component of $S$. This can be achieved using an enhanced version of the above algorithm, that uses `weight functions'. A \emph{weight function} is a function $w \colon [1,N] \rightarrow \mathbb{Z}^d$ that is constant on some interval in $[1,N]$ and is $0$ elsewhere. We will consider a finite collection of weight functions, all with the same value of $d$. The \emph{total weight} of a subset $A$ of $[1,N]$ is the sum, over all weight functions $w$ and all elements $x$ of $A$, of $w(x)$. For example, in the above decision problem that asks whether two specific points of $S \cap \mathcal{T}^1$ lie in the same component of $S$, we can set $d = 1$ and then use two weight functions, each of which is $1$ at one of the specified points of $S \cap \mathcal{T}^{1}$ and zero elsewhere. Thus the decision problem may be rephrased as: is there an equivalence class with total weight $2$? This can be answered using the following enhanced AHT algorithm:

\medskip
\emph{Extra input}: (3) a positive integer $d$;
\vspace{-0.05in}
\begin{itemize}
\item[(4)] a list of $\ell$ weight functions $[1,N] \rightarrow \mathbb{Z}^d$, each with range in $[-M,M]^d$.
\end{itemize}

\emph{Extra output}:
A list of equivalence classes and their total weights.

The running time of the algorithm is at most a polynomial function of $k$, $d$, $\ell$, $\log N$ and $\log M$.

Some of the uses of the AHT algorithm are as follows:

(1) \emph{Is a normal surface $S$ orientable?} To answer this, one counts the number of components of $S$ and the number of components of the surface with vector $2 (S)$. The latter is twice the former if and only if $S$ is orientable.

(2) \emph{How many boundary components does a properly embedded normal surface $S$ have?} Here, one considers just the points $\partial S \cap \mathcal{T}^{1}$ and just those pairings arising from normal arcs in $\partial S$. Then the vanilla version of AHT provides the number of components of $\partial S$.

(3) \emph{What are the components of $S$ as normal vectors?} Here, one sets $\ell$ and $d$ to be the number of edges of $\mathcal{T}$ and one defines the $i$th weight function $[1,N] \rightarrow \mathbb{Z}^d$ to take the value $(0, \dots, 0, 1, 0, \dots, 0)$, with the $1$ at the $i$th place, on precisely those points in $[1,N]$ that lie on the $i$th edge. So, the total weight of a component $S'$ just counts the number of intersection points between $S'$ and the various edges of $\mathcal{T}$. From this, one can readily compute $(S')$, as follows. The weights on the edges determine the points of intersection between $S'$ and the 1-skeleton. From these, one can compute the arcs of intersection between $S'$ and each face of $\mathcal{T}$. From these, one gets the curves of intersection between $S'$ and the boundary of each tetrahedron, and hence the decomposition of $S'$ into triangles and squares.

(4) \emph{What are the topological types of the components of $S$?} Using the previous algorithm one can compute the vectors for the components of $S$. Then for each component, one can compute its Euler characteristic, its number of boundary components and whether or not it is orientable. This determines its topological type.

\subsection{The algorithm}

This proceeds by modifying $[1,N]$, the pairings and the weights. At each stage, there will be a bijection between the old equivalence classes and the new ones that will preserve their total weights.

Although this is not how Agol-Hass-Thurston described their algorithm, it is illuminating to think of the pairings in terms of 2-complexes, in the following way.

We already view $[1,N]$ as a subset of $\mathbb{R}$. Each pairing can be viewed as specifying a band $[0,1] \times [0,1]$ attached onto $\mathbb{R}$, as follows. If the pairing is $[a,b] \rightarrow [c,d]$, then we attach $[0,1] \times \{ 0 \}$ onto $[a , b ]$ and we attach $[0,1] \times \{ 1 \}$ onto $[c , d ]$. There are two possible ways to attach the band: with or without a half twist, according to whether the pairing is orientation-preserving or reversing. (See Figure \ref{Fig:AgolHassThurston1}.)

If a pairing $[a,b] \rightarrow [c,d]$ satisfies $a \leq c$, then $[a,b]$ is the \emph{domain} of the pairing and $[c,d]$ is the \emph{range}. Its \emph{translation distance} is $c - a = d - b$. Its \emph{width} is $b - a + 1$, the number integers in its domain or its range.

\begin{figure}
  \includegraphics[width=1.8in]{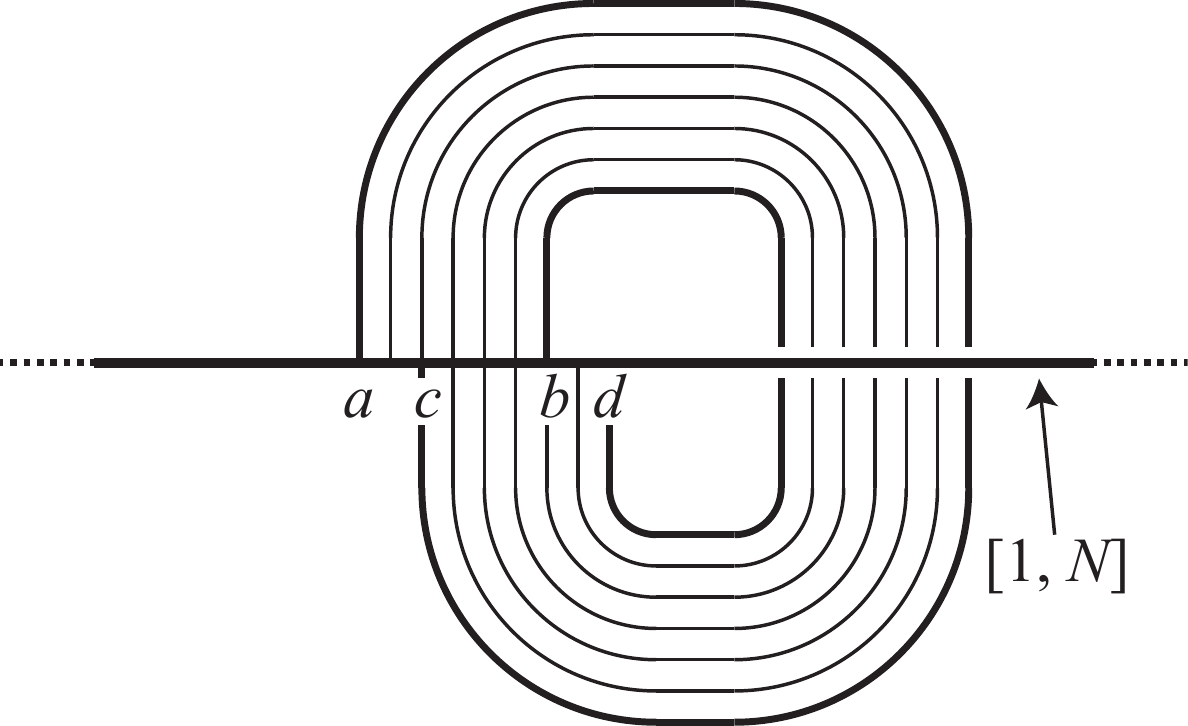}
   \hspace{.1in}
  \includegraphics[width=2in]{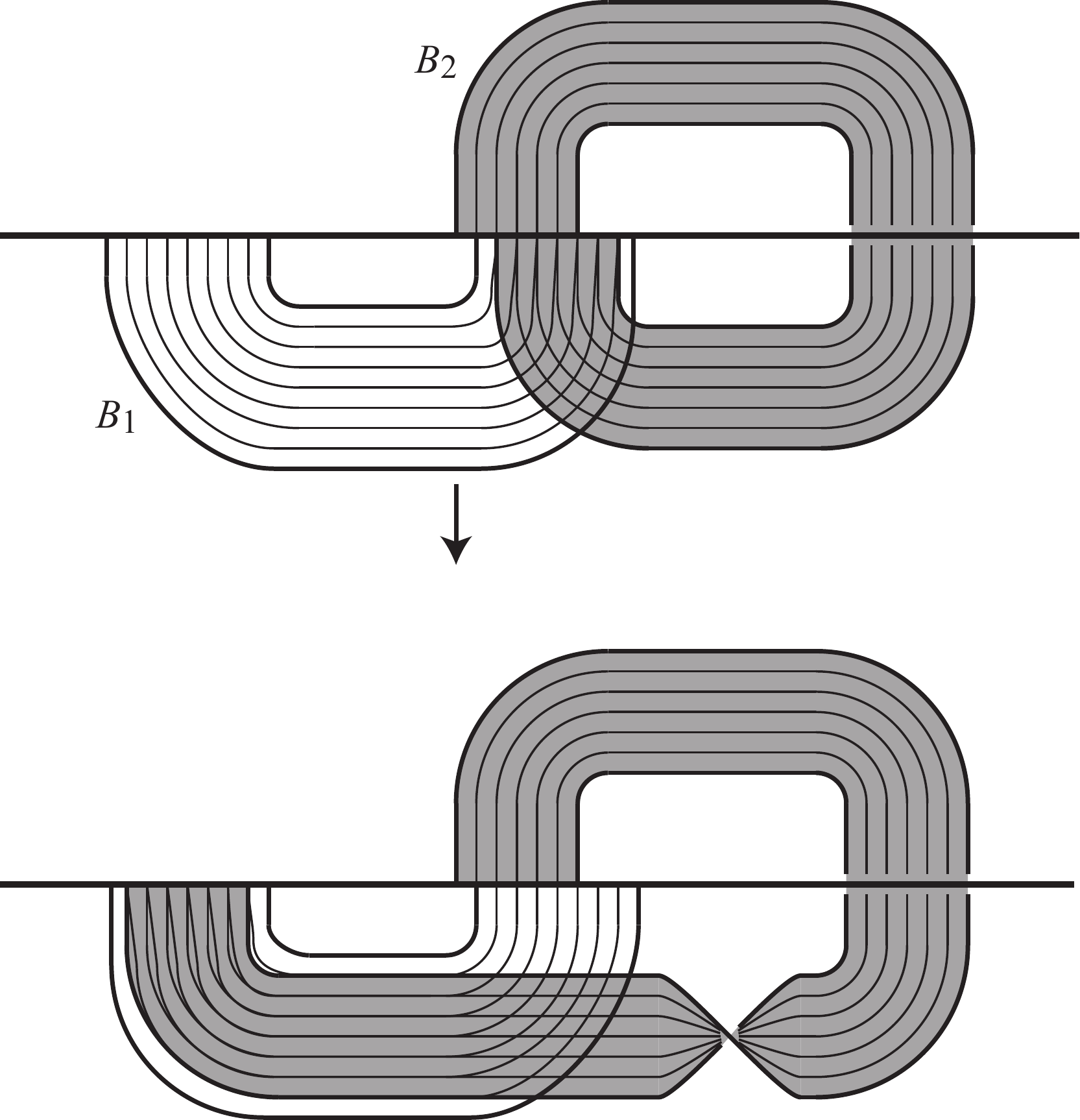}
  \caption{Left: Constructing a band from a pairing. Right: Transmission}
  \label{Fig:AgolHassThurston1}
\end{figure}

The modifications that AHT make can be viewed as alterations to these bands. The modifications will be chosen so that they reduce 
$4^k \prod_{i=1}^k w_i$,
where $k$ is the number of pairings and $w_1, \dots, w_k$ are their widths. Since the total running time of the algorithm is at most a polynomial function of $\log N$, it needs to be the case that at frequent intervals, this measure of complexity goes down substantially. In fact, after every $5k$ cycles through the following modifications, it will be the case that this quantity is scaled by a factor of at most $1/2$.

\emph{Transmission.} 
Suppose that one pairing $g_2$ has range contained within the range of another pairing $g_1$. Then one of the attaching intervals for the band $B_2$ corresponding to $g_2$ lies within an attaching interval for the band $B_1$ corresponding to $g_1$. Suppose also that the domain of $g_2$ is not contained in the range of $g_1$. Then modify $g_2$, by sliding the band $B_2$ along $B_1$, as shown in the right of Figure \ref{Fig:AgolHassThurston1}. On the other hand, if the domain of $g_2$ is contained in the range of $g_1$, then we can slide both endpoints of the band $B_2$ along the band $B_1$.

In fact, it might be possible to slide $B_2$ multiple times over $B_1$ if the domain and range of $g_1$ overlap. In this case, we do this as many times as possible, so as to move the domain and range of $B_2$ as far to the left as possible.

This process is used to move the attaching locus of the bands more and more to the left along the line $[1,N]$. Thus, the following modification might become applicable.

\emph{Truncation.}
Suppose that there is an interval $[c,N]$ that is incident to a single band. Then one can reduce the size of this band, or eliminate it entirely if the interval $[c,N]$ completely contains one of the attaching intervals of the band.

\begin{figure}
  \includegraphics[width=1in]{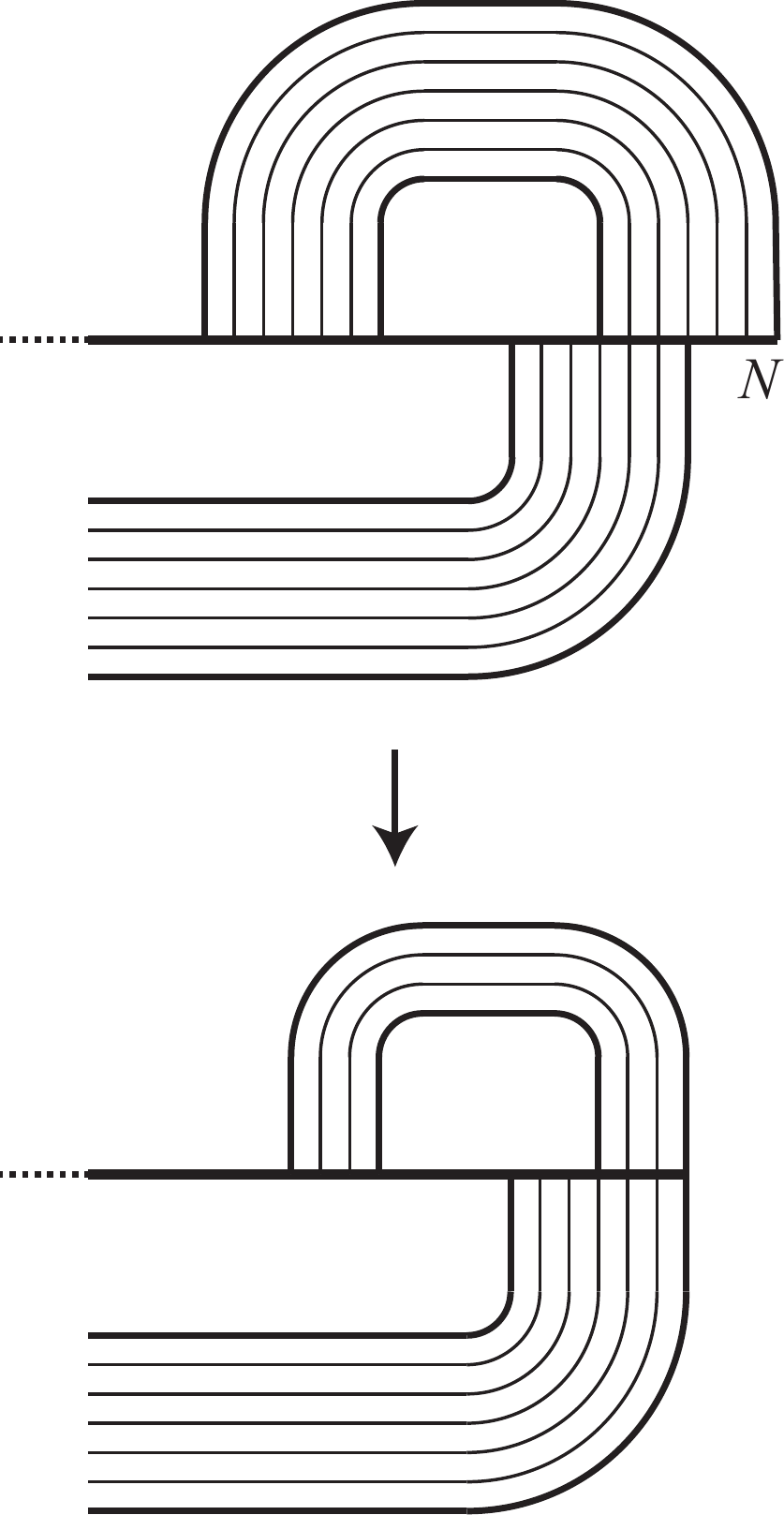}
   \hspace{.1in}
  \includegraphics[width=1.9in]{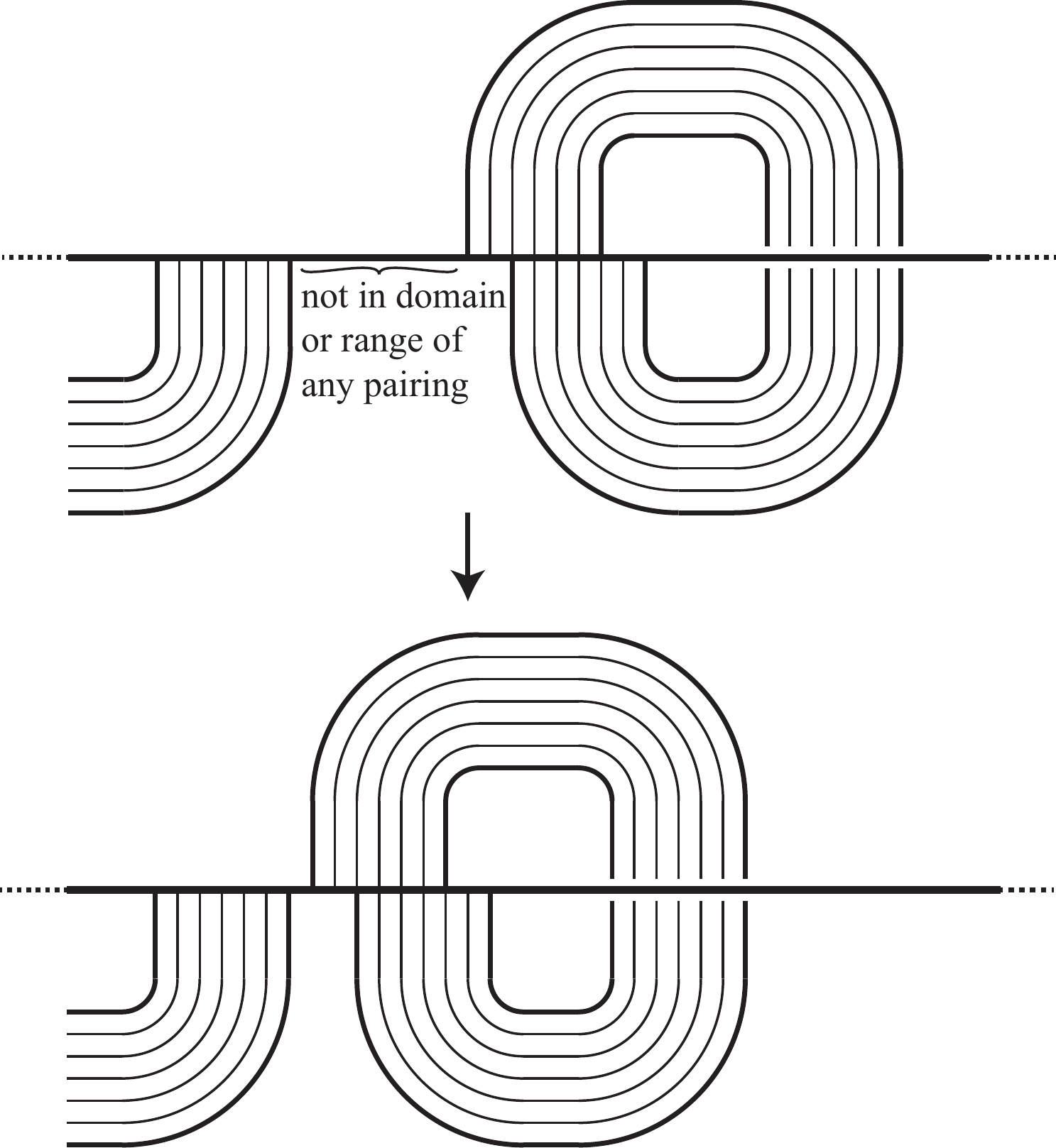}
   \hspace{.1in}
  \includegraphics[width=1.7in]{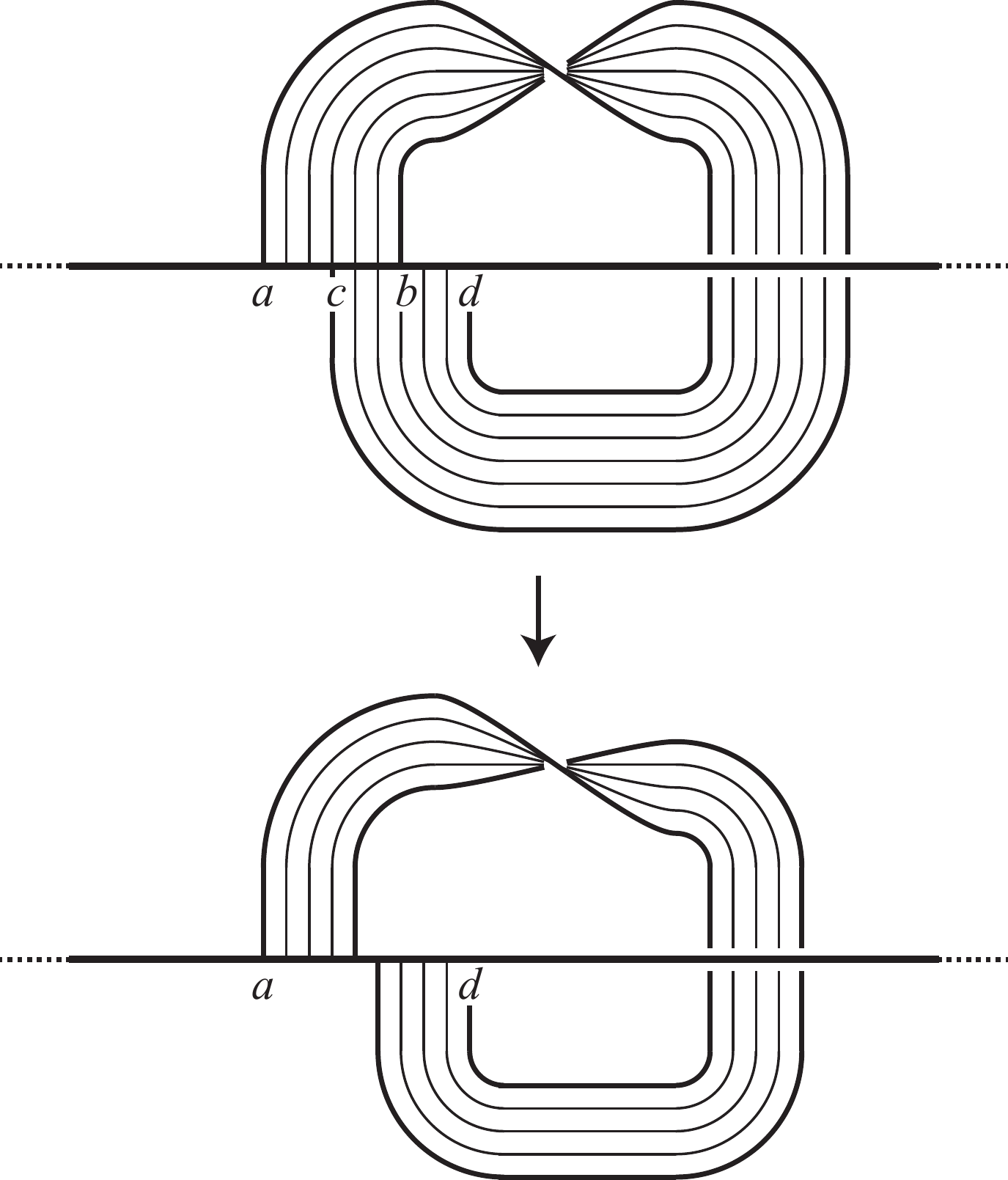}
  \caption{Left: Truncation. Middle: Contraction. Right: Trimming}
  \label{Fig:AgolHassThurston2}
\end{figure}

This process reduces the total width of the bands. It might also reduce the number of bands. Hence, the following might become applicable.

\emph{Contraction.}
Suppose that there is an interval in $[1,N]$ that is attached to no bands. Then each point in this interval is its own equivalence class. Thus, the procedure removes these points, and records them and their weights. 

\emph{Trimming.} 
Suppose that $g \colon [a,b] \rightarrow [c,d]$ is an orientation-reversing pairing with domain and range that overlap. Then trimming is the restriction of the domain and range of this pairing to $[a, \lceil (a+d)/2 \rceil - 1]$ and $[\lfloor (a+d)/2 \rfloor + 1, d]$ so that they no longer overlap.

As transmissions are performed, the attaching loci of the bands are moved to the left and so the two attaching intervals of a band are more likely to overlap. Under these circumstances, the pairing $g \colon [a,b] \rightarrow [c,d]$ is said to be \emph{periodic} if its domain and range intersect and it is also orientation-preserving. The combined interval $[a,d]$ is called a \emph{periodic interval}.

\emph{Period merger.}
When there are periodic pairings $g_1$ and $g_2$, then they can be replaced by a single periodic pairing, as long as there is sufficient overlap between their periodic intervals. Moreover, if $t_1$ and $t_2$ are their translation distances, then the new periodic pairing has translation distance equal to the greatest common divisor of $t_1$ and $t_2$.

Thus, with periodic mergers, it is possible to see very dramatic decrease in the widths of the bands. These also reduce the number of bands.

The AHT algorithm cycles through these modifications. The proof that it scales $4^k \prod_{i=1}^k w_i$ by a factor of at most $1/2$ after $5k$ cycles through the above steps is very plausible, but the proof in \cite{AgolHassThurston} is somewhat delicate.

\subsection{$3$-manifold knot genus is in NP}
\label{Sec:GenusNPHard}
One of the main motivations for Agol, Hass and Thurston to introduce their algorithm was to prove one half of Theorem \ref{Thm:3ManifoldKnotGenus}. Specifically, they used it to show that the problem of deciding whether a knot $K$ in a compact orientable 3-manifold $M$ bounds a compact orientable embedded surface of genus $g$ is in NP. The input is a triangulation $\mathcal{T}$ of $M$ with $K$ as a specified subcomplex, and an integer $g$ in binary. The first stage of the algorithm is to remove an open regular neighbourhood $N(K)$ of $K$, forming a triangulation $\mathcal{T}'$ of the exterior $X$ of $K$. If $K$ bounds a compact orientable surface of genus $g$, then it bounds such a surface, with possibly smaller genus, that is incompressible and boundary-incompressible in $X$. There is such a surface $S$ in normal form in $\mathcal{T}'$, by a version of Theorem \ref{Thm:EssentialNormal}. In fact, we may also arrange that $\partial S$ intersects each edge of $\mathcal{T}'$ at most once, by first picking $\partial S$ to be of this form and then noting that none of the normalisation moves in Section \ref{Sec:NormalSurfaces} affects this property. We may assume that $S$ has least possible weight among all orientable normal spanning surfaces with this boundary. By a version of Theorem \ref{Thm:EssentialSummand}, it is a sum of fundamental surfaces, none of which is a sphere or disc. In fact, it cannot be the case that two of these surfaces have non-empty boundary, because then $\partial S$ would intersect some edge of $\mathcal{T}'$ more than once. Hence, some fundamental summand $S'$ is also a spanning surface for $K$. It turns out that $S'$ must be orientable, as otherwise it is possible to find an orientable spanning surface with smaller weight. Note also that $\chi(S') \geq \chi(S)$ and hence the genus of $S'$ is at most the genus of $S$. The required certificate is the vector of $S'$. Since $S'$ is fundamental, we have a bound on its weight by Theorem \ref{Thm:WeightFundamental}. Using the AHT algorithm, one can easily check that it is connected, orientable and has Euler characteristic at least $1-2g$. One can also check that it has a single boundary curve that has winding number one along $N(K)$. 

\section{Showing that problems are NP-hard}
\label{Sec:NPHard}

In the previous section, we explained how the AHT algorithm can be used to show that the problem of deciding whether a knot in a compact orientable 3-manifold bounds a compact orientable surface of genus $g$ is in NP. Agol, Hass and Thurston also showed that this problem is NP-hard. Combining these two results, we deduce that the problem is NP-complete. In this section, we explain how NP-hardness is proved. Partly for the sake of variety, we will show that the following related problem is NP-hard, thereby establishing one half of the following result. This is minor variation of \cite[Theorem 1.1]{LackenbyConditionallyHard} due to the author.

\begin{theorem}
\label{Thm:LinkGenus}
The following problem is NP-complete. The input is a diagram of an unoriented link $L$ in the 3-sphere and a natural number $g$. The output is a determination of whether $L$ bounds a compact orientable surface of genus $g$.
\end{theorem}

Like the argument of Agol, Hass and Thurston, the method of proving that this is $\mathrm{NP}$-hard is to reduce this problem to an NP-complete problem, which is a variant of SAT, called 1-in-3-SAT. In SAT, one is given a list of Boolean variables $v_1, \dots, v_n$ and a list of sentences involving the variables and the connectives AND, OR and NOT. The problem asks whether there is an assignment of TRUE or FALSE to the Boolean variables that makes each sentence true. The set-up for 1-in-3-SAT is the same, except that the sentences are of a specific form. Each sentence involves exactly three variables or their negations, and it asks whether exactly one of them is true. An example is ``$v_1 \veebar \neg v_2 \veebar v_3$'' which means ``exactly one of $v_1$, $\mathrm{NOT} (v_2)$ and $v_3$ is true". Unsurprisingly, given its similarity to SAT, this problem is NP-complete \cite[p.~259] {GareyJohnson}

From a collection of 1-in-3-SAT sentences, one needs to create a diagram of a suitable link $L$. This is probably best described by means of the example in Figure \ref{Fig:LinkGenusNP}. Here, the variables are $v_1$, $v_2$ and $v_3$ and the sentences are
$$v_1 \veebar \neg v_2 \veebar v_3, \qquad v_1 \veebar \neg v_1 \veebar v_3, \qquad \neg v_1 \veebar \neg v_2 \veebar \neg v_3.$$
The associated link diagram is constructed from these as in Figure \ref{Fig:LinkGenusNP}. The parts of the diagram with a $K$ in a box are where the link forms a satellite of a knot $K$. This knot $K$ is chosen to have reasonably large genus: at least $2m+n$, where $m$ is the number of sentences and $n$ is the number of variables.

One can view this diagram as built as follows. Start with $n+1$ round circles in the plane of the diagram, where $n$ is the number of variables. Thus, each of the variables corresponds to one of the circles, and there is an extra circle. In Figure \ref{Fig:LinkGenusNP}, the first $n$ circles are arranged at the top of the figure and the extra circle is at the bottom. Also start with $4m$ components, where $m$ is the number of sentences. These are arranged into batches, each containing 4 link components that are 4 parallel copies of the knot $K$. Each sentence corresponds to a batch. Each sentence contains three variables. Given such a sentence, we attach a band from three out of the four components of the batch onto the three relevant variable components. If the negation of the variable appears in the sentence, we insert a half twist into the band. The fourth component of the batch is banded onto the extra component without a twist. 

The resulting link has $n+1$ components. We view an assignment of TRUEs and FALSEs to the variables as corresponding to a choice of orientation on the first $n$ components of the link. For example, in Figure \ref{Fig:LinkGenusNP}, the orientations shown correspond to the assignment of TRUE to $v_1$ and $v_2$ and the assignment of FALSE to $v_3$. The extra component is oriented in a clockwise way. These orientations determine orientations of the 4 strings within each batch. It should be evident that each sentence is true if and only if, within the corresponding batch, two of the strings are oriented one way and two of the strings are oriented the other. We call this a \emph{balanced} orientation. Thus, the given instance of 1-in-3-SAT has a solution if and only if the link has a balanced orientation.

\begin{figure}
  \includegraphics[width=4in]{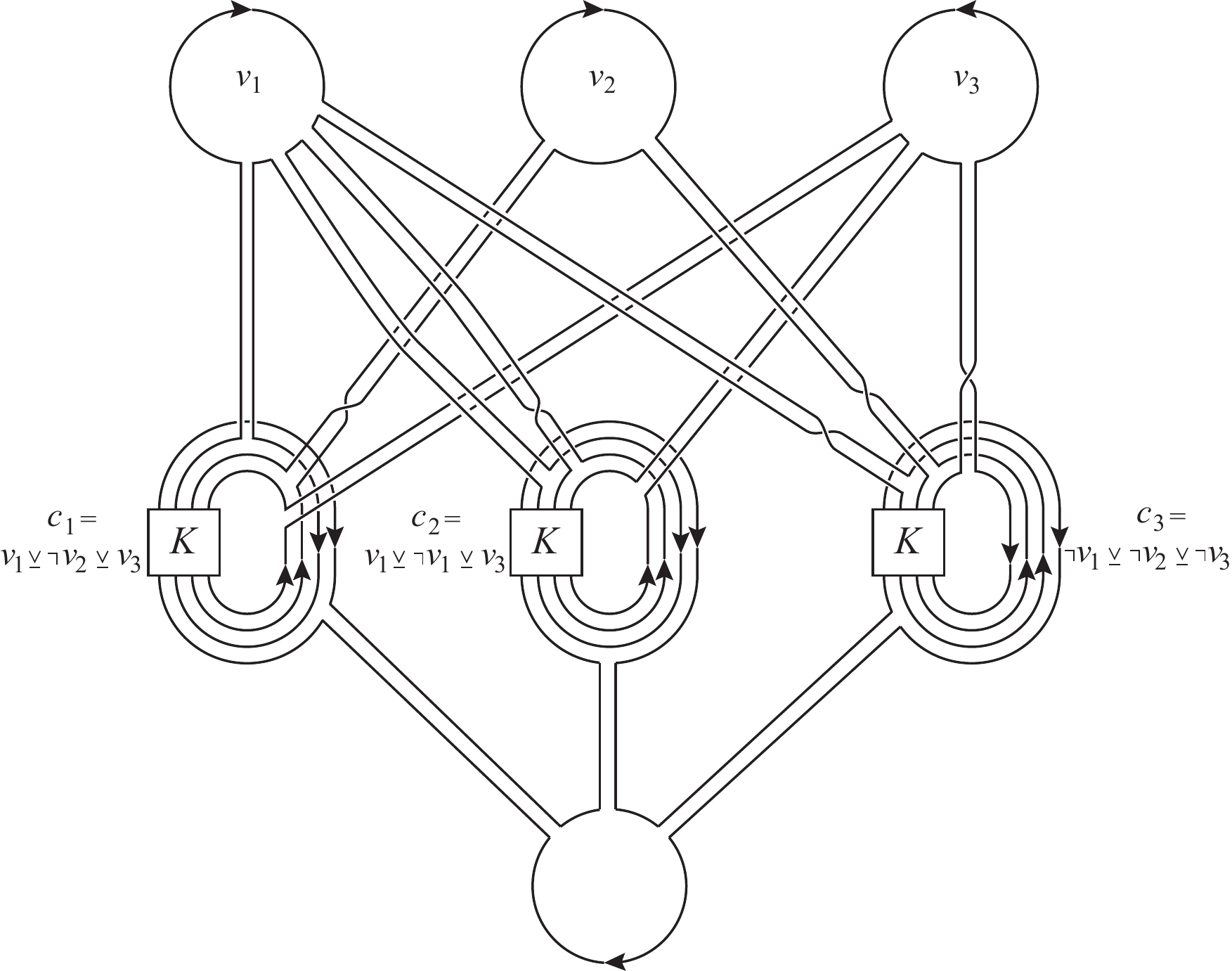}
  \caption{The link diagram obtained from an instance of 1-IN-3-SAT. The given orientation is a balanced one.}
  \label{Fig:LinkGenusNP}
\end{figure}

The NP-hardness of the decision problem in Theorem \ref{Thm:LinkGenus} is established by the following fact: the link $L$ has a balanced orientation if and only if it bounds a compact orientable surface of genus at most $2m$.

Suppose that it has a balanced orientation. Then it bounds the following surface. Start with a disc for each of the variable circles and the extra component. Within each batch, insert two annuli so that the boundary components of each annulus are oriented in opposite ways. Then attach bands joining the annuli to the discs. It is easy to check that the resulting surface is orientable and has genus at most $2m$.

Suppose that it does not have a balanced orientation. Then, for any orientation on the components of $L$, the resulting link is a satellite of the knot $K$ with non-zero winding number. Hence, by our assumption about the genus of $K$, the genus of any compact orientable surface bounded by $L$ is at least $2m+1$.

Thus, the NP-hardness of the problem in Theorem \ref{Thm:LinkGenus} is established. In fact, in \cite{LackenbyConditionallyHard}, a variant of this was established, which examined not the genus of a spanning surface but its Thurston complexity. (See Definition \ref{Def:ThurstonComplexity}.) But exactly the same argument gives Theorem \ref{Thm:LinkGenus}.

The fact that this problem is in NP is essentially the same argument as in Section \ref{Sec:GenusNPHard}.

\section{Hierarchies}
\label{Sec:Hierarchies}

So far, the theory that we have been discussing has mostly been concerned with a single incompressible surface. However, some of the most powerful algorithmic results are proved using sequences of surfaces called hierarchies.

\begin{definition} A \emph{partial hierarchy} for a compact orientable 3-manifold $M$ is a sequence of 3-manifolds $M = M_1, \dots, M_{n+1}$ and surfaces $S_1, \dots, S_{n}$, with the following properties:
\begin{enumerate}
\item Each $S_i$ is a compact orientable incompressible surface properly embedded in $M_i$.
\item Each $M_{i+1}$ is $M_i \cut S_i$.
\end{enumerate}
It is a \emph{hierarchy} if $M_{n+1}$ is a collection of 3-balls.
\end{definition}

The following was proved by Haken \cite{HakenHomeomorphism}.

\begin{theorem} If a compact orientable irreducible 3-manifold contains a  properly embedded orientable incompressible surface, then it admits a hierarchy. In particular, any compact orientable irreducible 3-manifold with non-empty boundary admits a hierarchy.
\end{theorem}

\begin{definition}
A compact orientable irreducible 3-manifold containing a compact orientable properly embedded incompressible surface is known as \emph{Haken}.
\end{definition}

Using hierarchies, Haken was able to prove the following algorithmic result \cite{HakenHomeomorphism}. (See Definition \ref{Def:FibreFree} for the definition of `fibre-free'.)

\begin{theorem}
\label{Thm:HakenHomeoProblem}
There is an algorithm to decide whether any two fibre-free Haken 3-manifolds are homeomorphic. 
\end{theorem}

In fact, using the solution to the conjugacy problem for mapping class groups of compact orientable surfaces \cite{Hemion}, it is possible to remove the fibre-free hypothesis.

An important part of the theory is the following notion.

\begin{definition}
A \emph{boundary pattern} $P$ for a 3-manifold $M$ is a subset of $\partial M$ consisting of disjoint simple closed curves and trivalent graphs.
\end{definition}

The following extension of Theorem \ref{Thm:HakenHomeoProblem} also holds.

\begin{theorem} 
\label{Thm:AlgorithmBoundaryPattern}
There is an algorithm that takes, as its input, two fibre-free Haken 3-manifolds $M$ and $M'$ with boundary patterns $P$ and $P'$ and determines whether there is a homeomorphism $M \rightarrow M'$ taking $P$ to $P'$.
\end{theorem}

Boundary patterns are used in the proof of Theorem \ref{Thm:HakenHomeoProblem}. However they are also useful in their own right. For example, when $L$ is a link in the 3-sphere and $M$ is the exterior of $L$, then it is natural to assign a boundary pattern $P$ consisting of a meridian curve on each boundary component. Suppose that $(M,P)$ and $(M',P')$ are the 3-manifolds and boundary patterns arising in this way from links $L$ and $L'$. Then there is a homeomorphism between $(M,P)$ and $(M',P')$ if and only if $L$ and $L'$ are equivalent links. Thus, we obtain the following immediate consequence of Theorem \ref{Thm:AlgorithmBoundaryPattern}

\begin{theorem}
There is an algorithm to decide whether any two non-split links in the 3-sphere are equivalent, provided their exteriors are fibre-free.
\end{theorem}

In fact, it is not hard to remove the non-split hypothesis from the above statement by first expressing a given link as a distant union of non-split sublinks. Also, as mentioned above, one can remove the fibre-free hypothesis, and thereby deal with all links in the 3-sphere.

A partial hierarchy determines a boundary pattern $P_i$ on each of the manifolds $M_i$, as follows. Either the initial manifold $M$ comes with a boundary pattern $P$ or $P$ is declared to be empty. The union $P \cup \partial S_1 \cup \dots \partial S_{i-1}$ forms a graph embedded in $M$. Then $P_i$ is defined to be the intersection between this graph and $M_i$. Provided $S_{i-1}$ is separating in $M_{i-1}$, provided $P_{i-1}$ is a boundary pattern and provided $\partial S_{i-1}$ intersects $P_{i-1}$ transversely and avoids its vertices, then $P_{i}$ is also a boundary pattern.

Note that in this definition, the surfaces are `transparent', in the following sense. Suppose that a boundary curve of $S_2$ runs over the surface $S_1$. Then we get boundary pattern in $\partial M_3$ on the other side of $S_1$, as well as at the intersection curves between the parts of $\partial M_3$ coming from $S_1$ and the parts coming from $S_2$. Thus, in total this curve of $\partial S_2$ gives rise to \emph{three} curves of $P_3$.

The key observation in the proof of Theorem \ref{Thm:AlgorithmBoundaryPattern} is that a hierarchy for $M$ induces a cell structure on $M$, as follows. The 1-skeleton is $P \cup \partial S_1 \cup \dots \cup \partial S_{n}$. The 3-cells are the components of $M_{n+1}$. The 2-cells arise where the components of $\partial M_{n+1} \cut P_{n+1}$ are identified in pairs and where the components of $\partial M_{n+1} \cut P_{n+1}$ intersect $\partial M$. There is a small chance that this is not a cell complex, because $\partial M_{n+1} \cut P_{n+1}$ might not be discs, but for essentially any reasonable hierarchy this will be the case. We say that this is the cell structure that is \emph{associated with the hierarchy}.

\begin{definition}
A hierarchy $H$ for $M$ is \emph{semi-canonical} if, given any triangulation of $M$, there is an algorithm to build a finite list of hierarchies for $M$, one of which is $H$.
\end{definition}

\begin{theorem}
\label{Thm:SemiCanonical}
Let $M$ be a fibre-free Haken 3-manifold with incompressible boundary. Then $M$ has a semi-canonical hierarchy.
\end{theorem}

The proof of Theorem \ref{Thm:AlgorithmBoundaryPattern} now proceeds as follows. Let $\mathcal{T}$ and $\mathcal{T}'$ be triangulations of fibre-free Haken 3-manifolds $M$ and $M'$. For simplicity, we will assume that they have incompressible boundaries. By Theorem \ref{Thm:SemiCanonical}, $M$ and $M'$ admit semi-canonical hierarchies $H$ and $H'$. If $M = M'$, then we may set $H = H'$. Thus, there is an algorithm to produce finite lists $H_1, \dots, H_m$ and $H'_1, \dots, H'_{m'}$ of hierarchies for each manifold, one of which is $H$ and one of which is $H'$. For each of these hierarchies, build the associated cell structure. If $M$ is homeomorphic to $M'$, there is therefore a cell-preserving homeomorphism from one of the cell structures for $M$ to one of these cells structures for $M'$. Conversely, if there is a cell-preserving homeomorphism, then clearly $M$ is homeomorphic to $M'$. Thus, the algorithm proceeds by searching for such a cell-preserving homeomorphism, which is clearly a finite task. 

We now explain, in outline, how the semi-canonical hierarchy in Theorem \ref{Thm:SemiCanonical} is constructed.

At each stage, we have a 3-manifold $M_i$ with a boundary pattern $P_i$, and we need to find a suitable surface $S_i$ to cut along. We suppose that $\partial M_i \setminus P_i$ is incompressible in $M_i$, which we can ensure as long as the initial manifold has incompressible boundary.  The choice of the surface $S_i$ depends on whether $(M_i, P_i)$ is simple, which is defined as follows.

\begin{definition}
\label{Def:SimplePattern}
Let $M$ be a compact orientable 3-manifold with a boundary pattern $P$. Then $(M, P)$ is \emph{simple} if the following conditions all hold:
\begin{enumerate}
\item $M$ is irreducible;
\item $\partial M \setminus P$ is incompressible;
\item any incompressible torus in $M$ is boundary parallel;
\item any properly embedded annulus disjoint from $P$ either is compressible, or admits a boundary compression disc disjoint from $P$, or is parallel to an annulus $A'$ in $\partial M$ such that $A' \cap P$ is a collection of disjoint core curves of $A'$.
\end{enumerate}
\end{definition}

The following procedure is then followed:
\begin{enumerate}
\item Suppose that $M_i$ contains an incompressible torus that is not boundary parallel. Then $S_i$ is either one or two copies of this torus, depending on whether the torus is separating or non-separating in the component of $M_i$ that contains it. 
\item Suppose that $M_i$ contains an annulus $A$ that is incompressible, disjoint from $P_i$ and not boundary-parallel, and that has the property that any other such annulus can be isotoped so that its boundary is disjoint from $\partial A$. Then $S_i$ is one or two copies of $A$, again depending on whether $A$ is separating or non-separating.
\item Suppose that $M_i$ contains an annulus $A$ disjoint from $P_i$ that is incompressible and that is parallel to annulus $A'$ in $\partial M_i$. Provided $A' \cap P_i$ is non-empty and does not just consist of core circles of $A'$, then $S_i$ is set to be $A$.
\item Suppose that $(M_i, P_i)$ has a simple component that is not a 3-ball. Suppose also that this component is not a solid torus with a longitude disjoint from $P_i$. Then by a version of Theorem \ref{Thm:ListSurfacesForSimpleManifold} for manifolds with boundary patterns, one may construct all incompressible boundary-incompressible surfaces in this component with maximal Euler characteristic and that intersect $P_i$ as few times as possible. 
We set $S_i$ to be one of these surfaces, or possibly two parallel copies of this surface, again depending on whether this surface is separating or non-separating.
\end{enumerate}

Of course, it is not clear why this procedure terminates. Moreover, this discussion is a substantial oversimplification. In particular, there are several more operations that are performed other than the ones described in (1) to (4) above. A careful and detailed account of the argument is given in Matveev's book \cite{Matveev}.

The reason for the fibre-free hypothesis is as follows. Suppose that $M$ fibres over the circle, for example, with empty boundary pattern. Then we might take the first surface $S_1$ to be a fibre. The next manifold $M_2$ is then a copy of $S_1 \times I$, with pattern $P_2$ equal to $\partial S_1 \times \partial I$. We note that none of the possibilities in the above procedure applies, and so there is no way to extend the partial hierarchy in a semi-canonical way. The manifold $(M_2, P_2)$ is not simple, but one is not allowed to cut along an annulus that is vertical in product structure and disjoint from the boundary pattern. One can show that if the above procedure is applied to a general Haken 3-manifold with incompressible boundary, and at some stage there is no possible $S_i$ to cut along, then in fact  the original manifold was not fibre-free.

The iterative nature of this algorithm makes it very inefficient. In order to construct the surface $S_i$, we need a triangulation $\mathcal{T}_i$ of $M_i$. This can be obtained from the triangulation $\mathcal{T}_{i-1}$ for $M_{i-1}$, as follows. First cut each tetrahedron along its intersection with $S_{i-1}$, forming a cell structure, and then subdivide this cell structure into a triangulation. The problem, of course, is that the number of normal triangles and squares of $S_{i-1}$ might be at least an exponential function of $|\mathcal{T}_{i-1}|$. Hence, $|\mathcal{T}_i|$ might be at least an exponential function of $|\mathcal{T}_{i-1}|$. Thus, the number of tetrahedra in these triangulations might grow like a tower of exponentials, where the height of the tower is the number of surfaces in the hierarchy. This is the source of the bounds in Theorem \ref{Thm:PachnerFibreFree} and Theorem \ref{Thm:UpperBoundRM}.

\section{The Thurston norm}
\label{Sec:ThurstonNorm}

We saw in Theorem \ref{Thm:3ManifoldKnotGenus} that the problem of deciding whether a knot in a compact orientable 3-manifold bounds a compact orientable surface with genus $g$ is NP-complete. Importantly, the 3-manifold is an input to the problem and so is allowed to vary. What if the manifold is fixed and only the knot can vary? Is the problem still NP-hard? We also saw in Theorem \ref{Thm:LinkGenus} that the problem of deciding whether a link in the 3-sphere bounds a compact orientable surface of genus $g$ is also NP-complete. Here, the background manifold is fixed. However, there is no requirement that the surface respects any particular orientation on the link. So the surface might represent one of many different homology classes in the link exterior. Indeed, a critical point in the proof was a choice of orientation of the link. What if we fix an orientation on the link at the outset and require the Seifert surface to respect this orientation? In particular, what if we focus on knots in the 3-sphere, rather than links, where there is a unique homology class (up to sign) for a Seifert surface?

Perhaps surprisingly, these restricted problems are almost certainly \emph{not} NP-complete. For example, we have the following result of the author \cite{LackenbyEfficientCertification}.

\begin{theorem}
\label{Thm:KnotGenusNPCoNP}
The problem of deciding whether a knot in the 3-sphere bounds a compact orientable surface of genus $g$ is in $\mathrm{NP}$ and $\mathrm{co}$-$\mathrm{NP}$.
\end{theorem}

We also saw in Theorem \ref{Thm:UnknotNPCoNP} that the problem of recognising the unknot is in NP and co-NP. Recall from Conjecture \ref{Con:NPCoNP} that such problems are believed not to be NP-complete.

Recent work of the author and Yazdi \cite{LackenbyYazdi} generalises Theorem \ref{Thm:KnotGenusNPCoNP} to knots in an arbitrary but fixed 3-manifold.

\begin{theorem}
\label{Thm:KnotGenusFixedManifold}
Let $M$ be a fixed compact orientable 3-manifold. Then the problem of deciding whether a knot in $M$ bounds a compact orientable surface of genus $g$ is in $\mathrm{NP}$ and $\mathrm{co}$-$\mathrm{NP}$.
\end{theorem}

We will discuss in this section the methods that go into the proof of Theorem \ref{Thm:KnotGenusNPCoNP}.

\subsection{The Thurston norm}

\begin{definition}
\label{Def:ThurstonComplexity}
The \emph{Thurston complexity} $\chi_-(S)$ for a compact connected surface $S$ is $\max \{ 0, -\chi(S) \}$. The Thurston complexity $\chi_-(S)$ of a compact surface $S$ with components $S_1, \dots, S_n$ is $\sum_i \chi_-(S_i)$.
\end{definition}

\begin{definition}
Let $M$ be a compact orientable 3-manifold. The \emph{Thurston norm} $x(z)$ of a class $z \in H_2(M, \partial M)$ is the minimal Thurston complexity of a compact oriented properly embedded surface representing $z$.
\end{definition}

Theorem \ref{Thm:KnotGenusNPCoNP} is a consequence of the following result.

\begin{theorem}
\label{Thm:ThurstonNormNP}
The following problem is in $\mathrm{NP}$. The input is a triangulation of a compact orientable 3-manifold $M$, a simplicial cocycle $c$ representing an element $[c]$ of $H^1(M)$ and a non-negative integer $n$ in binary. The size of the input is defined to be the sum of the number of tetrahedra in the triangulation, the number of digits of $n$ and the sum of the number of digits of $c(e)$ as $e$ runs over each edge of the triangulation. The output is an answer to the question of whether the Poincar\'e dual of $[c]$ has Thurston norm $n$.
\end{theorem}

\begin{proof}[Proof of Theorem \ref{Thm:KnotGenusNPCoNP}] Let $D$ be a diagram of a knot $K$ with $c(D)$ crossings and let $g$ be a natural number. Let $g(K)$ be the genus of $K$, which therefore lies between $0$ and $(c(D)-1)/2$.  If $g \geq g(K)$, then Theorem \ref{Thm:3ManifoldKnotGenus} provides a certificate, verifiable in polynomial time, that $K$ bounds a compact orientable surface of genus $g$. So suppose that $g < g(K)$. From the diagram, we can construct a triangulation $\mathcal{T}$ of the exterior $M$ of $K$ where the number of tetrahedra is bounded above by a linear function of $c(D)$. We can also give a cocycle $c$ representing a generator for $H^1(M)$. Theorem \ref{Thm:ThurstonNormNP} provides a certificate that the Thurston norm of the dual of $[c]$ is $\max \{ 2g(K) - 1, 0 \}$ and hence that the genus of $K$ is $g(K)$. This establishes that $K$ does not bound a compact orientable surface of genus $g$.
\end{proof}

Exactly the same argument establishes that unknot recognition lies in co-NP, which is one half of Theorem \ref{Thm:UnknotNPCoNP}.

\subsection{Sutured manifolds}

Theorem \ref{Thm:ThurstonNormNP} is proved using Gabai's theory of sutured manifolds \cite{Gabai}.

\begin{definition} 
A \emph{sutured manifold} is a compact orientable 3-manifold $M$ with two specified subsurfaces $R_-$ and $R_+$ of $\partial M$. These must satisfy the condition that $R_- \cup R_+ = \partial M$ and that $R_- \cap R_+ = \partial R_- = \partial R_+$. The subsurface $R_-$ is transversely oriented into $M$, and $R_+$ is transversely oriented outwards. The curves $\gamma = R_- \cap R_+$ are called \emph{sutures}. The sutured manifold is denoted $(M, \gamma)$.
\end{definition}

\begin{definition}
A compact oriented embedded surface $S$ in $M$ with $\partial S \subset \partial M$ is said to be \emph{taut} if it is incompressible and it has minimal Thurston complexity in its class in $H_2(M, \partial S)$.
\end{definition}

A basic example of a taut surface is a minimal genus Seifert surface in the exterior of a knot. Indeed, when $M$ has toral boundary, then a compact oriented properly embedded surface $S$ is taut if and only if it is incompressible and has minimal Thurston complexity in its class in $H_2(M, \partial M)$.

\begin{definition}
A sutured manifold $(M, \gamma)$ is \emph{taut} if $M$ is irreducible and $R_-$ and $R_+$ are both taut.
\end{definition}

\begin{definition}
Let $(M, \gamma)$ be a sutured manifold and let $S$ be a properly embedded transversely oriented surface that intersects $\gamma$ transversely. Then $M \cut S$ inherits a sutured manifold structure $(M \cut S, \gamma_S)$ as follows. There are two copies of $S$ in $\partial (M \cut S)$, one pointing inwards, one outwards. These form part of $R_-(M \cut S, \gamma_S)$ and $R_+(M \cut S, \gamma_S)$. Also, the intersection between $R_-(M, \gamma)$ and $M \cut S$ forms the remainder of the the inward pointing subsurface. The outward pointing subsurface is defined similarly. This is termed a \emph{sutured manifold decomposition} and is denoted
$$(M, \gamma) \xrightarrow{S} (M \cut S, \gamma_S).$$
\end{definition}

\begin{definition}
A \emph{sutured manifold hierarchy} is a sequence of decompositions 
$$(M_1, \gamma_1) \xrightarrow{S_1} (M_2, \gamma_2) \xrightarrow{S_2} \dots \xrightarrow{S_n} (M_{n+1}, \gamma_{n+1})$$
where each $(M_i, \gamma_i)$ is taut, each surface $S_i$ is taut and $(M_{n+1}, \gamma_{n+1})$ is a collection of taut 3-balls.
\end{definition}

The fundamental theorems of sutured manifold theory are the following surprising results \cite[Theorems 2.6, 3.6 and 4.19]{Scharlemann}

\begin{theorem}
\label{Thm:TautnessPullsBack}
Let 
$$(M_1, \gamma_1) \xrightarrow{S_1} (M_2, \gamma_2) \xrightarrow{S_2} \dots \xrightarrow{S_n} (M_{n+1}, \gamma_{n+1})$$
be a sequence of sutured manifold decompositions with the following properties for each surface $S_i$:
\begin{enumerate}
\item no component of $\partial S_i$ bounds a disc in $\partial M_i$ disjoint from $\gamma_i$;
\item no component of $S_i$ is a compression disc for a solid toral component of $M_i$ with no sutures.
\end{enumerate}
Suppose that $(M_{n+1}, \gamma_{n+1})$ is taut. Then every sutured manifold $(M_i, \gamma_i)$ is taut and every surface $S_i$ is taut.
\end{theorem}

\begin{theorem}
\label{Thm:SuturedHierarchiesExist}
Let $(M, \gamma)$ be a taut sutured manifold and let $z \in H_2(M, \partial M)$ be a non-trivial class. Then there is a sutured manifold hierarchy for $(M, \gamma)$ satisfying (1) and (2) in Theorem \ref{Thm:TautnessPullsBack} and where the first surface $S_1$ satisfies $[S_1] = z$.
\end{theorem}

Such a hierarchy can be viewed as certificate for the tautness of the first surface $S_1$ and hence for the Thurston norm of $[S_1]$ in the case where $\partial M$ is empty or toral. Note that the tautness of the final manifold $(M_{n+1}, \gamma_{n+1})$ is easily verified. This is because a 3-ball with a sutured manifold structure is taut if and only it has at most one suture.

\subsection{A certificate verifiable in polynomial time} Theorems \ref{Thm:TautnessPullsBack} and \ref{Thm:SuturedHierarchiesExist} imply that sutured manifold hierarchies can be used to establish the Thurston norm of a homology class in $H_2(M, \partial M)$, provided $\partial M$ is empty or toral and $M$ is irreducible. However, it is somewhat surprising that they can be used to form a certificate that is verifiable in polynomial time. Indeed, the discussion at the end of Section \ref{Sec:Hierarchies} suggests that it is hard to control the complexity of hierarchies.

However, sutured manifold hierarchies seem to be much more tractable than ordinary ones. As we will see, the reason for this is that there is an important distinction between the behaviour of sutures and boundary patterns when a manifold is cut along a surface.

Recall that the main source of the complexity of hierarchies is that a fundamental normal surface $S$ in a compact orientable 3-manifold $M$ may have exponentially many triangles and squares, as a function of the number of tetrahedra in the triangulation $\mathcal{T}$ of $M$. Thus, when we attempt to build a triangulation of $M \cut S$, we may need exponentially many tetrahedra. However, there are at most $5$ different triangle and square types that can coexist within a tetrahedron $\Delta$ of $\mathcal{T}$. Hence, all but at most $6$ components of $\Delta \cut S$ lie between parallel normal discs. These regions patch together to form an $I$-bundle embedded in $M \cut S$ called its \emph{parallelity bundle} \cite{LackenbyCrossing}. Thus, $M \cut S$ is composed of at most $6|\mathcal{T}|$ bits of tetrahedra with the parallelity bundle attached to them. Even when $S$ is exponentially complicated, it is possible to determine the topological types of the components of the parallelity bundle in polynomial time using the AHT algorithm. Hence, in fact, $M \cut S$ is not as complicated as it first seems.

It would be ideal if the next two stages of the hierarchy after $S$ consisted of the annuli that form the vertical boundary of the parallelity bundle, and then vertical discs in the $I$-bundle that decompose it to balls. Then the resulting manifold would have a simple triangulation. In the case of Haken's hierachies in Section \ref{Sec:Hierarchies}, this is not possible. It is not permitted to decompose along vertical annuli in an $I$-bundle that are disjoint from the boundary pattern. In fact, before an $I$-bundle can be decomposed in Haken's hierarchies, its horizontal boundary must first receive non-empty boundary pattern, from decompositions along surfaces elsewhere in the manifold. 

However, in the case of sutured manifolds, these sort of decompositions are allowed, under some fairly mild hypotheses. In particular, a decomposition along an incompressible annulus is always permitted, provided one boundary component lies in $R_-$ and one lies in $R_+$. It is therefore possible, after these decompositions, to obtain a triangulation of the resulting manifold with a controlled number of tetrahedra. In this way, we may build the entire sutured manifold hierarchy for $M$ and encode it in a way that makes it possible to verify in polynomial time that the final manifold consists of taut balls and that it satisfies (1) and (2) of Theorem \ref{Thm:TautnessPullsBack}. This is the basis for the author's proof \cite{LackenbyEfficientCertification} of Theorem \ref{Thm:ThurstonNormNP}.

\section{3-sphere recognition and Heegaard splittings}
\label{Sec:AlmostNormal}

In groundbreaking work \cite{Rubinstein}, Rubinstein proved the following fundamental algorithmic result.

\begin{theorem}
There is an algorithm to determine whether a 3-manifold is the 3-sphere.
\end{theorem}

This result is remarkable because the 3-sphere is quite featureless and so, unlike the case of unknot recognition, there is not an obvious normal surface to search for. Rubinstein's argument was enhanced and simplified by Thompson \cite{Thompson}. Both arguments relied on the theory of almost normal surfaces, which are defined as follows.

\begin{definition}
A surface properly embedded in a tetrahedron is
\begin{itemize}
\item an \emph{octagon} if it is a disc with boundary consisting of eight normal arcs;
\item a \emph{tubed piece} if it is an annulus that is obtained from two disjoint normal discs by attaching a tube that runs parallel to an edge of the tetrahedron.
\end{itemize}
An octagon or tubed piece is called an \emph{almost normal piece}.
\end{definition}

\begin{figure}
  \includegraphics[width=3.5in]{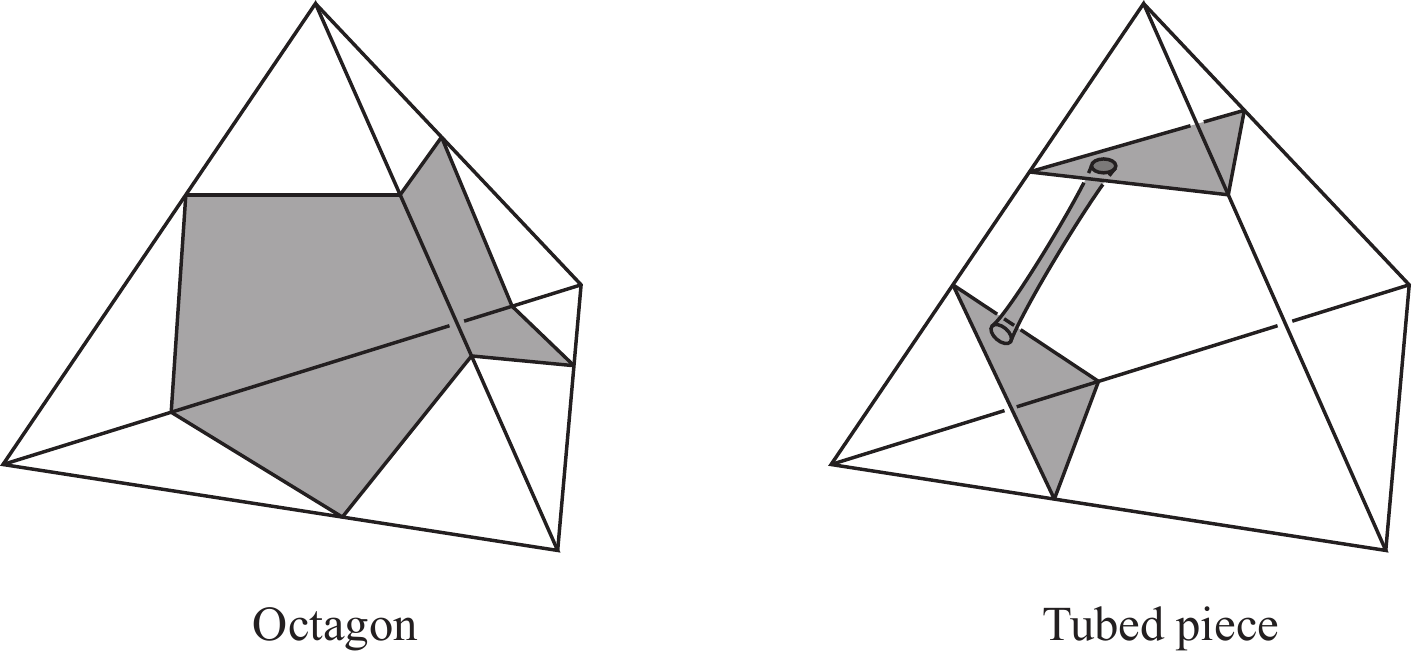}
  \caption{Almost normal pieces}
  \label{Fig:AlmostNormal}
\end{figure}

\begin{definition}
\label{Def:AlmostNormalSurface}
A surface properly embedded in a triangulated 3-manifold is \emph{almost normal} if it intersects each tetrahedron in a collection of disjoint triangles
and squares, except in precisely one tetrahedron where it consists of exactly one almost normal piece and possibly also some triangles and squares.
\end{definition}

The following striking result is the basis for the Rubinstein-Thompson algorithm.

\begin{theorem}
\label{Thm:AlmostNormal2Sphere}
Let $\mathcal{T}$ be a triangulation of a closed orientable 3-manifold $M$. Suppose that $\mathcal{T}$ has a single vertex and contains no normal spheres other the one consisting of triangles surrounding the vertex. Then $M$ is the 3-sphere if and only if $\mathcal{T}$ contains an almost normal embedded 2-sphere.
\end{theorem}

The hypothesis that $\mathcal{T}$ has a single vertex and that it has a unique normal 2-sphere may sound restrictive, but in fact one may always build such a triangulation for a closed orientable irreducible 3-manifold $M$ unless $M$ is one of three exceptional cases: the 3-sphere, $\mathbb{RP}^3$ and the lens space $L(3,1)$. (See \cite{Burton:Crushing} for example.)

Both directions of Theorem \ref{Thm:AlmostNormal2Sphere} are remarkable. Suppose first that $\mathcal{T}$ contains an almost normal 2-sphere. One of the features of an almost normal surface $S$ is that it admits an obvious isotopy that reduces the weight of $S$. This isotopy moves the surface off $S$ along an edge compression disc (see (6) at the end of Section \ref{Sec:NormalSurfaces} for the definition of an edge compression disc). When the surface contains an octagon, then there are two choices for the isotopy, one in each direction away from $S$. When the surface contains a tubed piece, then there is just one possible direction. We then continue to apply these weight reducing isotopies, all going in the same direction. It turns out this continues to be possible until the resulting surface is normal or is a small 2-sphere lying in a single tetrahedron (see \cite{Schleimer} or \cite{Mijatovic3Sphere}). In the case where we get a normal 2-sphere, then this is, by hypothesis, the boundary of a regular neighbourhood of the vertex of the triangulation. Since we only ever isotope in the same direction, the image of the isotopy is homeomorphic to $S \times [0,1]$. Thus, we deduce that the manifold is $S \times [0,1]$ with small 3-balls attached. Hence, it is a 3-sphere.

Suppose now that $M$ is the 3-sphere. Then one removes a small regular neighbourhood of the vertex to get a 3-ball $B$. This has a natural height function $h \colon B \rightarrow [0,1]$, just given by distance from the origin of the ball. The key to the argument is to place the 1-skeleton of $\mathcal{T}$ into `thin position' with respect to this height function. This notion, which was originally due to Gabai \cite{Gabai3}, is defined as follows. We may assume that the restriction of $h$ to $\mathcal{T}^1$ has only finitely critical points, which are local minima or local maxima. Let $0 < x_1 < \dots < x_n < 1$ be their values under $h$. Then the number of intersection points between $\mathcal{T}^1$ and sphere $h^{-1}(t)$ for $t \in (x_i, x_{i+1})$ is some constant $c_i$. This remains true if we set $x_0 = 0$ and $x_{n+1} = 1$. We say that the 1-skeleton $\mathcal{T}^1$ is in \emph{thin position} if $\sum_i c_i$ is minimised. It turns out in this situation, a value of $c_i$ that is maximal gives rise to a surface $h^{-1}(t)$ that is nearly almost normal. More specifically, there is a sequence of compressions in the complement of the 1-skeleton that takes it to an almost normal surface.


Theorem \ref{Thm:AlmostNormal2Sphere} was the basis for the following result of Ivanov \cite{Ivanov} and Schleimer \cite{Schleimer}.

\begin{theorem}
Recognising whether a 3-manifold is the 3-sphere lies in $\mathrm{NP}$.
\end{theorem}

The certificate was essentially just the almost normal 2-sphere, although they had to deal with possible presence of normal 2-spheres also.

In the above implication that if $M$ is the 3-sphere, then it contains an almost normal 2-sphere, the fact that it is a 2-sphere plays very little role. What matters is that the 3-sphere has a `sweepout' by 2-spheres, starting with a tiny 2-sphere that encircles the vertex of $\mathcal{T}$ and ending with a small 2-sphere surrounding some other point. Such sweepouts arise in another natural situation: when $M$ is given via a Heegaard splitting. 

\begin{definition} A \emph{compression body} $C$ is either a handlebody or is obtained from $F \times [0,1]$, where $F$ is a (possibly disconnected) compact orientable surface, by attaching 1-handles to $F \times \{ 1 \}$. The \emph{negative boundary} is the copy of $F \times \{ 0 \}$ or empty (in the case of a handlebody). The \emph{positive boundary} is the remainder of $\partial C$. A \emph{Heegaard splitting} for a compact orientable 3-manifold $M$ is an expression of $M$ as a union of two compression bodies glued by a homeomorphism between their positive boundaries. The resulting \emph{Heegaard surface} is the image of the positive boundaries in $M$.
\end{definition}

A handlebody can be viewed as a regular neighbourhood of a graph, as follows. It is a 0-handle with 1-handles attached. The graph is a vertex at the centre of the 0-handle together with edges, each of which runs along a core of a 1-handle. It is known as a \emph{core} of the handlebody. Similarly, a compression body that is not a handlebody is a regular neighbourhood of its negative boundary together with some arcs that start and end on the negative boundary and run along the cores of the 1-handles. These are known as \emph{core arcs} of the compression body. Thus, given any Heegaard splitting for $M$, there is an associated graph $\Gamma$ in $M$, which is the cores of the two compression bodies. Then $M \setminus (\Gamma\cup \partial M)$ is a copy of $S \times (-1,1)$, where $S \times \{ 0 \}$ is the Heegaard surface. Thus, the projection map $S \times (-1,1) \rightarrow (-1,1)$ extends to a function $h \colon M \rightarrow [-1,1]$ and we can place $\mathcal{T}^1$ into thin position with respect to this height function.

Using these methods, Stocking \cite{Stocking} proved the following result, based on arguments of Rubinstein.

\begin{theorem}
Let $\mathcal{T}$ be a triangulation of a compact orientable 3-manifold $M$. Let $S$ be a Heegaard surface for $M$ that is strongly irreducible, in the sense that any compression disc for $S$ on one side necessarily intersects any compression disc for $S$ on the other. Suppose that it is not a Heegaard torus for the 3-sphere or the 3-ball. Then there is an ambient isotopy taking $S$ into almost normal form.
\end{theorem}

This alone is not enough to be able to solve many algorithmic problems about Heegaard splittings. This is because there seems to be no good way of ensuring that the almost normal surface $S$ is a bounded sum of fundamental surfaces. However, the author was able to use it to prove the following result \cite{LackenbyHeegaard}, in combination with methods from hyperbolic geometry.

\begin{theorem}
There is an algorithm to determine the minimal possible genus of a Heegaard surface for a compact orientable simple 3-manifold with non-empty boundary.
\end{theorem}

\section{Homomorphisms to finite groups}
\label{Sec:Homomorphisms}

In recent years, homomorphisms from 3-manifold groups to finite groups have been used to prove some important algorithmic results. The following theorem of Kuperberg \cite{KuperbergKnottedness}, which was proved using these techniques, still remains particularly striking.

\begin{theorem} 
\label{Thm:KnottednessGRH}
The problem of deciding whether a knot in the 3-sphere is non-trivial lies in $\mathrm{NP}$, assuming the Generalised Riemann Hypothesis.
\end{theorem}

This result has been superseded by Theorem \ref{Thm:UnknotNPCoNP}, which removes the conditionality on  the Generalised Riemann Hypothesis. However, the techniques that Kuperberg introduced remain important. In particular, they were used by Zentner to show that the problem of deciding whether a 3-manifold is the 3-sphere lies in co-NP, assuming the Generalised Riemann Hypothesis. Kuperberg's argument is explained in Section \ref{Subsec:Unknot3Sphere}.

\subsection{Residual finiteness}

\begin{definition} A group $G$ is \emph{residually finite} if, for every $g \in G$ other than the identity, there is a homomorphism $\phi$ from $G$ to a finite group, such that $\phi(g)$ is non-trivial.
\end{definition}

\begin{theorem} 
\label{Thm:ResiduallyFinite3Manifold}
Any compact orientable 3-manifold has residually finite fundamental group.
\end{theorem}

This was proved by Hempel \cite{Hempel} for manifolds that are Haken, but it now applies to all compact orientable 3-manifolds, as a consequence of Hempel's work and Perelman's solution to the Geometrisation Conjecture \cite{Perelman1, Perelman2, Perelman3}. We will discuss the proof below.

Residual finiteness has long been known to have algorithmic implications. For example, when a finitely presented group $G$ is residually finite, it has solvable word problem. The argument goes as follows. Suppose we are given a finite presentation of $G$ and a word $w$ in the generators and their inverses. It is true for any group that if $w$ is the trivial element, then there is an algorithm that eventually terminates with a proof that it is trivial. For one may start enumerating all products of conjugates of the relations, and thereby start to list all words in the generators that represent the trivial element. If $w$ is trivial, it will therefore eventually appear on this list. On the other hand, if $w$ is non-trivial, then by residual finiteness, there is a homomorphism $\phi$ to a finite group such that $\phi(w)$ is non-trivial. Thus, one can start to enumerate all finite groups and all homomorphisms $\phi$ from $G$ to these groups and eventually one will find a $\phi$ such that $\phi(w)$ is non-trivial. By running these two processes in parallel, we will eventually be able to decide whether a given word represents the identity element.

Theorem \ref{Thm:ResiduallyFinite3Manifold} is fairly straightforward for any compact hyperbolic 3-manifold $M$. In this case, the hyperbolic structure gives an injective homomorphism $\pi_1(M) \rightarrow \mathrm{Isom}^+(\mathbb{H}^3) = \mathrm{SO}(3,1)$ and it is therefore linear. The following theorem of Malcev \cite{Malcev} then applies.

\begin{theorem} Any finitely generated linear group is residually finite.
\end{theorem}

It is instructive to consider a specific example before sketching the proof of the general theorem. Let $M$ be a Bianchi group, for example $\mathrm{PSL}(2,\mathbb{Z}[i])$. Now $\mathbb{Z}[i] = \mathbb{Z}[t] / \langle t^2 + 1 \rangle$. One may form finite quotients of this ring by quotienting by the ideal $\langle m \rangle$ for some positive integer $m$. The result is the finite ring $\mathbb{Z}_m[t] /  \langle t^2 + 1 \rangle$. Thus, we obtain a homomorphism
$$\phi_m \colon \mathrm{PSL}(2,\mathbb{Z}[i]) \rightarrow \mathrm{SL}(2,\mathbb{Z}_m[t] / \langle t^2 + 1 \rangle) / \{ \pm I \}.$$
Now consider any non-trivial element $g$ of $\mathrm{PSL}(2,\mathbb{Z}[i])$. This is not congruent to $\pm I$ modulo $m$ for all $m$ sufficiently large. Thus, for any such $m$, the image of $g$ under $\phi_m$ is non-trivial. We have therefore proved that $\mathrm{PSL}(2,\mathbb{Z}[i])$ is residually finite.

This generalises to any finitely generated group $G$ that is linear over a field $k$, as follows. One considers a finite generating set $g_1, \dots, g_t$ for $G$. Then $g_1^{\pm 1}, \dots, g_t^{\pm 1}$ correspond to matrices. The entries of these matrices generate a ring $R$ that is a subring of $k$. The group $G$ therefore lies in $\mathrm{GL}_n(R)$. It is possible to show that any such ring $R$ has a collection of finite index ideals $I_m$ such that $\bigcap_m I_m = \{ 0 \}$. Thus, given any non-trivial element $g$ of $G$, its image in the finite group $\mathrm{GL}_n(R/I_m)$ is non-trivial for some $m$.

The above argument was only for compact hyperbolic manifolds $M$. More generally, any compact geometric 3-manifold has linear fundamental group. However, it is not currently known whether the fundamental group of every compact orientable 3-manifold is linear. Instead, Hempel proved his theorem by using the decomposition of a compact orientable 3-manifold $M$ along spheres and tori into geometric pieces. The fundamental group of each piece has residually finite fundamental group. By considering carefully the finite quotients of these groups, Hempel was able to show that these homomorphisms could be chosen to be compatible along the JSJ tori. Hence, the prime summands of the manifold are residually finite. This gives the theorem because it is a general result that a free product of residually finite groups is residually finite.

\subsection{Unknot and 3-sphere recognition}
\label{Subsec:Unknot3Sphere}

Kuperberg's Theorem \ref{Thm:KnottednessGRH} was proved by using some of the above methods in a quantified way.  Let $M$ be the exterior of a non-trivial knot $K$ in the 3-sphere. Then $\pi_1(M)$ is non-abelian. This can be proved either by appealing to the Geometrisation Conjecture or by using Theorem 9.13 in \cite{HempelBook} that classifies the 3-manifold groups that are abelian. By the residual finiteness of $\pi_1(M)$, there is a finite quotient of $\pi_1(M)$ that is non-abelian. This can be seen by noting that if $g$ and $h$ are non-commuting elements of $\pi_1(M)$, then there is some homomorphism $\phi$ to a finite group such that $\phi([g,h])$ is non-trivial. The image of this homomorphism is therefore non-abelian. The key claim in Kuperberg's proof is that this non-abelian finite quotient of $\pi_1(M)$ can be chosen so that it has controlled size (as a function of the size of the given input, which might be a diagram of $K$ or a triangulation of the exterior of $K$). Thus, this quotient can used as a certificate, verifiable in polynomial time, of the non-triviality of $K$.

To get control over the size of this finite quotient of $\pi_1(M)$, Kuperberg uses the theory of linear groups. However, as mentioned above, it is not known that every 3-manifold group is linear; this is not known even for the exteriors of knots in the 3-sphere. But Kuperberg observed that we do not need the full strength of linearity to make the above argument work. All we need to know is that there is some non-abelian quotient of $\pi_1(M)$ that is linear. This is provided by the following result of Kronheimer and Mrowka \cite{KronheimerMrowka} that is proved using the theory of instantons.

\begin{theorem} 
\label{Thm:KronheimerMrowka}
Let $K$ be any non-trivial knot in the 3-sphere. Then there is a homomorphism $\pi_1(S^3 \setminus K) \rightarrow \mathrm{SU}(2)$ with non-abelian image.
\end{theorem}

Then once one has this linear representation, the existence of a homomorphism to a finite group with non-abelian image is a consequence. Assuming the Generalised Riemann Hypothesis, one can get the following control over the size of this group \cite[Theorem 3.4]{KuperbergKnottedness}.

\begin{theorem} 
\label{Thm:AffineAlgebraicQuotient}
Let $G$ be an affine algebraic group over $\mathbb{Z}$. Let $\Gamma$ be a group with a presentation where the sum of the lengths of the relations is $\ell$. Suppose that there is a homomorphism $\Gamma \rightarrow G(\mathbb{C})$ with non-abelian image. Then, assuming the Generalised Riemann Hypothesis, there is a homomorphism $\Gamma \rightarrow G(\mathbb{Z}/p)$ with non-abelian image, for some prime $p$ with $\log p$ bounded above by a polynomial function of $\ell$.
\end{theorem}

Rather than giving the precise definition of an affine algebraic group and the terminology $G(\mathbb{C})$ and $G(\mathbb{Z}/p)$, we focus on the key example of $\mathrm{SL}(2 ,\mathbb{C})$, which gives the general idea. Note that $\mathrm{SU}(2) \subset \mathrm{SL}(2 ,\mathbb{C})$ and hence Theorem \ref{Thm:KronheimerMrowka} also gives a homomorphism into $\mathrm{SL}(2 ,\mathbb{C})$ with non-abelian image.

Now, $\mathrm{SL}(2 ,\mathbb{C})$ is an algebraic subvariety of $\mathbb{C}^{4}$, since the condition that a matrix has determinant one is a polynomial equation. The coefficients of the polynomial are integers. We write this group as $G(\mathbb{C})$. Thus, for any positive integer $k$, one can define the group $G(\mathbb{Z}/k) = SL(2, \mathbb{Z}/k)$ as a subset of $(\mathbb{Z}/k)^{4}$ with the same defining equation.

An outline of the proof of Theorem \ref{Thm:AffineAlgebraicQuotient} is as follows. One considers all homomorphisms $\Gamma \rightarrow G(\mathbb{C})$. To define such a homomorphism, one need only specify where the generators of $\Gamma$ are mapped and check that each of the relations map to the identity. Thus, each homomorphism determines a point in $\mathbb{C}^{4t}$, where $t$ is the number of generators. The set of all such points is an algebraic subvariety because the relations in the group impose polynomial constraints. We are interested in homomorphisms into $G(\mathbb{C})$ with non-abelian image, and it is in fact possible to view this subset also as an algebraic variety in $\mathbb{C}^n$, for some $n > 4t$.

By assumption, this variety is non-empty. It is a well known fact that any affine variety in $\mathbb{C}^n$ defined using polynomials with integer coefficients contains a point whose coordinates are algebraic numbers. Koiran \cite{Koiran} quantified this result by expressing such a point as
$$(x_1, \dots, x_{n}) = (g_1(\alpha), \dots, g_n(\alpha)),$$
where $g_1, \dots, g_n$ are polynomials with integer coefficients and $\alpha$ is a root of an irreducible integer polynomial $h$, with control over the degree and the size of the coefficients of the polynomials.

It is now that the Generalised Riemann Hypothesis is used. It implies that the polynomial $h(x)$ also has a root $r$ in $\mathbb{Z}/p$ for some prime $p$ with bounded size. In fact, $\log p$ ends up being at most a polynomial function of $\ell$, the sum of the lengths of the relations of $\Gamma$. Thus, $(g_1(r), \dots, g_n(r))$ is a point in $(\mathbb{Z}/p)^n$ that corresponds to a homomorphism $\Gamma \rightarrow G(\mathbb{Z}/p)$ with non-abelian image.

\begin{proof}[Proof of Theorem \ref{Thm:KnottednessGRH}] Let $K$ be a non-trivial knot in the 3-sphere, given via its diagram or a triangulation $\mathcal{T}$ of its exterior. In the former case, we build a triangulation $\mathcal{T}$ of the exterior. This triangulation can easily be used to build a presentation of $\pi_1(S^3 \setminus K)$ with length $\ell$, which is at most a linear function of $|\mathcal{T}|$. By Theorem \ref{Thm:KronheimerMrowka}, there is a homomorphism $\pi_1(S^3 \setminus K) \rightarrow \mathrm{SL}(2, \mathbb{C}) = G(\mathbb{C})$ with non-abelian image. Hence, by Theorem \ref{Thm:AffineAlgebraicQuotient}, there is a homomorphism $\pi_1(S^3 \setminus K) \rightarrow G(\mathbb{Z}/p)$ with non-abelian image, where $\log p$ is at most a polynomial function of $\ell$. This homomorphism provides a certificate, verifiable in polynomial time of the non-triviality of $K$.
\end{proof}

Exactly the same proof strategy was used by Zentner \cite{Zentner} to show that 3-sphere recognition lies in co-NP, assuming GRH. In this case though, the major new input was the following result of Zentner.

\begin{theorem}
Let $M$ be a homology $3$-sphere other than the 3-sphere. Then $\pi_1(M)$ admits a homomorphism to $\mathrm{SL}(2, \mathbb{C})$ with non-abelian image.
\end{theorem}

The first step in the proof of this uses a theorem of Boileau, Rubinstein and Wang \cite{BRW}. This asserts that such a $3$-manifold admits a degree one map onto a hyperbolic 3-manifold, a Seifert fibre space other than $S^3$ or a space obtained by gluing together the exteriors of two non-trivial knots in 3-sphere, by identifying the meridian of each with the longitude of the other.  This latter space is called the \emph{splice} of the two knots. A degree one map between 3-manifolds induces a surjection between their fundamental groups. So it suffices to focus on the case where $M$ is one of the above three possibilities.
When $M$ is hyperbolic, it is an almost direct consequence of the definition that it admits a faithful homomorphism to $\mathrm{SL}(2, \mathbb{C})$. When $M$ is Seifert fibred, it is not hard to find  a representation into $\mathrm{SU}(2) \subset \mathrm{SL}(2, \mathbb{C})$ with non-abelian image. Thus, the difficult case is the splice of two knots. Zenter proves his theorem in this situation using instantons, as in the case of Theorem \ref{Thm:KronheimerMrowka}.

\section{Hyperbolic structures}
\label{Sec:HyperbolicStructures}

As stated in Theorem \ref{Thm:HomeoProblem}, the homeomorphism problem for compact orientable 3-manifolds is solved. All known solutions use the Geometrisation Conjecture, which asserts that any compact orientable 3-manifold has `a decomposition into geometric pieces'. The most important and ubiquitous pieces are the hyperbolic ones. Therefore in this section, we discuss the solution to the homeomorphism problem for hyperbolic 3-manifolds. The solution for general compact orientable 3-manifolds uses the techniques in this section, as well as an algorithmic construction of the decomposition of a manifold into its prime summands, and a construction of the pieces of its JSJ decomposition. Our presentation is based on Kuperberg's paper \cite{KuperbergAlgorithmic}.

\begin{theorem}
\label{Thm:HyperbolicHomeoProblem}
There is an algorithm that takes as its input triangulations of two closed hyperbolic $3$-manifolds and determines whether these manifolds are homeomorphic.
\end{theorem}

\begin{definition}
A  triangulation of a closed hyperbolic 3-manifold $M$ is \emph{geodesic} if each simplex is totally geodesic.
\end{definition}

\begin{lemma}
\label{Lem:StraightExists}
Any closed hyperbolic 3-manifold admits a geodesic triangulation.
\end{lemma}

\begin{proof} Pick a point $p$ in the manifold $M$, and let $\tilde p$ be a point in its inverse image in $\mathbb{H}^3$. Let $\mathcal{P}$ be the set of points in $\mathbb{H}^3$ that are closer to $\tilde p$ than to any covering translate of $\tilde p$. Its closure is a finite-sided polyhedron. The vertices, edges and faces of this polyhedron project to $0$-cells, $1$-cells and $2$-cells of a cell complex in $M$. The remainder of $M$ is a 3-cell. Now subdivide this to a triangulation. Do this by placing a vertex in the interior of each 2-cell and coning off. Then cone the 3-cell from $p$. The result is a geodesic triangulation of $M$.
\end{proof}

An alternative way of viewing a geodesic triangulation is as a recipe for building a hyperbolic structure on $M$. One realises each tetrahedron as the convex hull of four points in hyperbolic space that do not lie in a plane. One has to to ensure that the face identifications between adjacent tetrahedra are realised by isometries. Of course, the angles around each edge sum to $2 \pi$. In fact, by Poincar\'e's polyhedron theorem \cite{EpsteinPetronio}, these conditions are enough to specify a hyperbolic structure on $M$.

This observation is key to the proof of the following result.


\begin{theorem}
\label{Thm:FindStraight}
There is an algorithm that takes as its input a triangulation $\mathcal{T}$ of some closed hyperbolic $3$-manifold $M$ and provides a sequence of Pachner moves taking it to a geodesic triangulation. Moreover, it provides the lengths of the edges in this geodesic triangulation as logarithms of algebraic numbers.
\end{theorem}

\begin{proof} 
We start applying all possible Pachner moves to $\mathcal{T}$, creating a list of triangulations of $M$. If this procedure were left to run indefinitely, it would create an infinite list of all triangulations of $M$. For each triangulation $\mathcal{T}'$, we start to try to find a geodesic structure on it. Thus, for each tetrahedron of $\mathcal{T}'$, we consider possible arrangements of four points in the upper-half space model for $\mathbb{H}^3$. We only consider points with co-ordinates that are algebraic numbers. As algebraic numbers can be enumerated, one can consider all such arrangements in turn. For each such arrangement, we compute the lengths of the edges and the interior angles at the edges. For each edge length $\ell$ and angle $\alpha$, $e^\ell$ and $e^{i \alpha}$ are in fact algebraic functions of the co-ordinates of the points in $\mathbb{H}^3$. We then check whether whenever two tetrahedra are glued along an edge, then these edge lengths are the same. We also check whether the angles around each edge sum to $2 \pi$. This is possible since these conditions are polynomial equations of the variables. (In fact, it is convenient to use polynomial inequalities here also \cite{KuperbergAlgorithmic}.) This fact also explains why we may restrict to co-ordinates that are algebraic numbers. This is because whenever a system of polynomial equations and inequalities with integer coefficients has a real solution, then it has one with algebraic co-ordinates.

Thus, for each of these triangulations, if it can be realised as geodesic, then one such geodesic structure will eventually be found. Hence, this process terminates. \end{proof}

The importance of geodesic triangulations is demonstrated in the following result.

\begin{theorem}
\label{Thm:PachnerBoundStraight}
Let $\mathcal{T}_1$ and $\mathcal{T}_2$ be geodesic triangulations of a closed hyperbolic 3-manifold. Then there is a computable upper bound on the number of Pachner moves required to pass from $\mathcal{T}_1$ to $\mathcal{T}_2$, as a function of the lengths of the edges of $\mathcal{T}_1$ and $\mathcal{T}_2$.
\end{theorem}

\begin{proof} The idea is to find a common subdivision $\mathcal{T}_3$ of $\mathcal{T}_1$ and $\mathcal{T}_2$, and to bound the number of Pachner moves joining $\mathcal{T}_1$ to $\mathcal{T}_3$ and joining $\mathcal{T}_3$ to $\mathcal{T}_2$. This subdivision $\mathcal{T}_3$ is obtained by superimposing $\mathcal{T}_1$ and $\mathcal{T}_2$, to form a cell structure $\mathcal{C}$, and then subdividing this to a triangulation. Before we do this, we possibly perturb $\mathcal{T}_2$ a little, maintaining it as a geodesic triangulation, so that it is in general position with respect to $\mathcal{T}_1$. Each 3-cell of $\mathcal{C}$ is a component of intersection between a tetrahedron of $\mathcal{T}_1$ and a tetrahedron of $\mathcal{T}_2$. To obtain $\mathcal{T}_3$, we cone off the 2-cells and the 3-cells of $\mathcal{C}$ (as described in the proof of Theorem \ref{Lem:StraightExists}). Now, the 3-cells come in finitely many combinatorial types, since the intersection between any two tetrahedra in $\mathbb{H}^3$ is a polyhedron with a bounded number of faces. Thus, the number of tetrahedra in $\mathcal{T}_3$ is controlled by the number of 3-cells in each tetrahedron of $\mathcal{T}_1$, say. In fact, it is not hard to prove that the number of Pachner moves taking $\mathcal{T}_1$ to $\mathcal{T}_3$ is also controlled by this quantity. Thus, to prove the theorem, it is necessary to obtain an upper bound on the number of times a tetrahedron of $\mathcal{T}_1$ and a tetrahedron of $\mathcal{T}_2$ can intersect. The triangulations $\mathcal{T}_1$ and $\mathcal{T}_2$ lift to geodesic triangulations $\tilde{\mathcal{T}}_1$ and $\tilde{\mathcal{T}}_2$ of $\mathbb{H}^3$. Any two geodesic tetrahedra in $\mathbb{H}^3$ are either disjoint or have connected intersection. Thus, we need to control the number of tetrahedra of $\tilde{\mathcal{T}}_2$ that can intersect a single tetrahedron of $\tilde{\mathcal{T}}_1$.

Let $\Delta_1$ and $\Delta_2$ be geodesic tetrahedra in $\tilde{\mathcal{T}}_1$ and $\tilde{\mathcal{T}}_2$, and let $\ell_1$ and $\ell_2$ be the maximal length of their sides. Let $p_1$ and $p_2$ be points in the interiors of $\Delta_1$ and $\Delta_2$. Suppose that a covering translate of $\Delta_2$ intersects $\Delta_1$. Then the corresponding covering translate of $p_2$ is at a distance at most $\ell_1 + \ell_2$ from $p_1$. Hence, the translate of $\Delta_2$ that contains it lies within the ball of radius $\ell_1 + 2\ell_2$ about $p_1$. Let $V$ be the volume of this ball. Let $v$ be the volume of $\Delta_2$, which can be determined from the edge lengths of $\Delta_2$. Then the number of translates of $\Delta_2$ that can intersect $\Delta_1$ is at most $V/v$. Hence, the number of tetrahedra in $\mathcal{T}_3$ is at most a constant times $|{\mathcal{T}}_1| \, |{\mathcal{T}}_2| (V/v)$. 
\end{proof}

\begin{proof}[Proof of Theorem \ref{Thm:HyperbolicHomeoProblem}] Let $\mathcal{T}_1$ and $\mathcal{T}_2$ be the given triangulations of closed hyperbolic 3-manifolds $M_1$ and $M_2$. By Theorem \ref{Thm:FindStraight}, we can find a sequence of Pachner moves taking $\mathcal{T}_1$ and $\mathcal{T}_2$ to geodesic ones $\mathcal{T}'_1$ and $\mathcal{T}'_2$. Theorem \ref{Thm:FindStraight} also provides the lengths of their edges. Hence, by Theorem \ref{Thm:PachnerBoundStraight}, we have a computable upper bound on the number of Pachner moves relating $\mathcal{T}'_1$ and $\mathcal{T}'_2$, if $M_1$ and $M_2$ are homeomorphic. We search through all sequences of moves with at most this length. Thus, if $M_1$ and $M_2$ are homeomorphic, we will find a sequence of Pachner moves relating $\mathcal{T}'_1$ and $\mathcal{T}'_2$. If we do not find such a sequence, then we know that $M_1$ and $M_2$ are not homeomorphic.
\end{proof}

It would be very interesting to have a quantified version of Theorem \ref{Thm:HyperbolicHomeoProblem}, which would provide an explicit upper bound on the number of Pachner moves relating any two triangulations $\mathcal{T}_1$ and $\mathcal{T}_2$ of a closed hyperbolic 3-manifold. Theorem \ref{Thm:PachnerBoundStraight} provides such a bound for geodesic triangulations. However, there is no obvious upper to the number of Pachner moves used in the algorithm for converting a given triangulation to a geodesic one that is presented in the proof of Theorem \ref{Thm:FindStraight}. The arguments given in \cite{KuperbergAlgorithmic} might be useful here for general triangulations of closed hyperbolic 3-manifolds. However, any prospect of getting a bound that is a polynomial function of $|\mathcal{T}_1|$ and $|\mathcal{T}_2|$ still seems to be a long way off.

\bibliographystyle{amsplain}
\bibliography{algorithmic-references-initials}

\end{document}